\documentclass[reqno]{amsart}

\usepackage{paralist, amsmath, amsthm, amssymb, color, graphicx, hyperref, tikz, amssymb}

\usepackage[shortlabels]{enumitem}
\usepackage{siunitx}
\usepackage{nicefrac}
\usepackage{dsfont}
\usepackage{mathrsfs}
\usepackage{biblatex}
\usepackage{changepage}
\DeclareFieldFormat{doi}{\href{https://dx.doi.org/#1}{\textsc{doi}}}

\theoremstyle{plain}
\newtheorem{theorem}{Theorem}[section]

\newtheorem{lemma}[theorem]{Lemma}
\newtheorem{corollary}[theorem]{Corollary}

\theoremstyle{definition}
\newtheorem{definition}[theorem]{Definition}
\newtheorem{remark}[theorem]{Remark}

\newtheorem*{notation}{Notation}

\newcommand{\al}{\alpha}

\newcommand{\eps}{\varepsilon}
\newcommand{\ga}{\gamma}

\newcommand{\mA}{\mathcal{A}}
\newcommand{\mB}{\mathcal{B}}

\newcommand{\N}{\mathbb{N}}

\newcommand{\R}{\mathbb{R}}
\newcommand{\Sn}{\mathfrak{S}}
\newcommand{\pd}{\textup{pd}}

\newcommand{\s}{\text{span}}
\newcommand{\proj}{\text{proj}}

\newcommand{\sse}{\subseteq}
\newcommand{\sli}{\sum\limits}

\newcommand{\emp}{\emptyset}
\newcommand{\ra}{\rightarrow}

\newcommand{\semicolon}{;}
\newcommand{\npm}{\pm [n]}

\DeclareMathOperator{\absmin}{absmin}

\title[]{Interpreting the (signed) chromatic polynomial coefficients via hyperplane arrangements}
\author{Neha Goregaokar}
\date{\today}

\address{Department of Mathematics, Brandeis University, Waltham, MA 02453, USA}
\email{ngoregaokar@brandeis.edu}

\begin{document}

\begin{abstract}
A recent result of Lofano and Paolini expresses the characteristic polynomial of a real hyperplane arrangement in terms of a projection statistic on the regions of the arrangement. We use this result to give an alternative proof for Greene and Zaslavsky's interpretation for the coefficients of the chromatic polynomial of a graph and further generalize this interpretation to signed graphs. We also show that this projection statistic has a nice combinatorial interpretation in the case of the braid arrangement, which generalizes to graphical arrangements of natural unit interval graphs. 
\end{abstract}

\maketitle

\section{Introduction}
A (real) hyperplane arrangement is a finite collection of affine hyperplanes in $\R^n$ for $n \geq 1$. These hyperplanes partition the space into convex regions. A central question of the study of hyperplane arrangements is counting the number of regions of a given hyperplane arrangement. One of the major results regarding this problem is from Zaslavsky's thesis~\cite{ZThesis}, where he proved that the number of regions of a hyperplane arrangement is (upto sign) the evaluation of the characteristic polynomial of that arrangement at $-1$. 

Now, given a hyperplane arrangement, one can associate to it a polynomial known as the characteristic polynomial. It is known that the coefficients of this polynomial are integral and alternate in sign. Hence, an evaluation of the characteristic polynomial at $-1$ is the sum of the absolute values of its coefficients. This raises the question of if one can express the characteristic polynomial of an arrangement $\mA \sse \R^n$ as a statistic on the set of regions of $\mA$. 

No such statistic was known until recently, when Lofano and Paolini~\cite{LP} and Kabluchko~\cite{Kabluchko} independently gave the same projection statistic. They showed that given a generic point $v \in \R^n$, the characteristic polynomial of $\mA$ can be expressed as the generating function of the regions counted by the dimension of the face containing the projection of $v$ to the region (see Section~\ref{bgsec} for the exact statement). 

For many hyperplane arrangements, there is a labeling of the regions by combinatorial objects. Hence the question arises: can we find a point $v\in \R^n$ such that the projection statistic given in \cite{ Kabluchko, LP} corresponds to a natural combinatorial statistic on the objects labeling the regions? In this paper, we answer this question for graphical arrangements and more generally, all subarrangements of the Type $B$ Coxeter arrangement. 

Given a graph $G = ([n],E)$ one can define the corresponding \textit{graphical arrangement} $\mA_G$ as the collection of the hyperplanes $x_i - x_j = 0$ where $\{i, j\} \in E$. It is known that for any graph $G$, 
\begin{equation}\label{CC}
    \chi_G(t) = \chi_{A_G}(t)
\end{equation}
that is, the chromatic polynomial of a graph is equal to the characteristic polynomial of the corresponding graphical arrangement. Hence, any interpretation of the coefficients of the chromatic polynomial is an interpretation of the coefficients of the characteristic polynomial. 
It is also known that each region of the graphical arrangement can be uniquely labeled by an acyclic orientation of $G$. This means that the chromatic polynomial can be expressed through~\cite{Kabluchko,LP} as a generating function on acyclic orientations of $G$ according to a projection statistic (which depends on choice of projection point $v$).   

\begin{figure}[ht]
    \centering
    \includegraphics[width=0.7\linewidth]{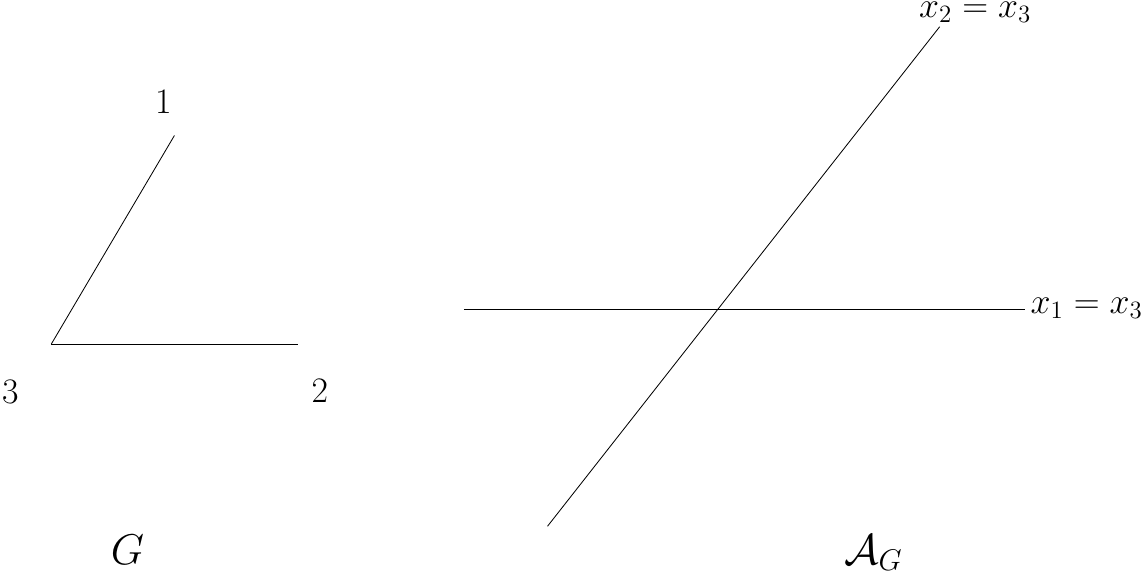}
    \caption{A graph $G$ and the corresponding graphical arrangement $\mA_G$.}
    \label{fig:IntroGraph}
\end{figure}
In~\cite{GZ}, Greene and Zaslavsky gave an interpretation for the coefficients of the chromatic polynomial of a graph $G$ as counting acyclic orientations of $G$ according to their number of source components (see Section~\ref{bgsec}). This raises the question of whether the Greene and Zaslavsky result can be related to the projection statistic, and whether there exists a point $v \in \R^n$ such that for each region of $\mA_G$, the projection statistic given in ~\cite{Kabluchko, LP} equals the number of source components of the acyclic orientation of $G$ labeling the region as given in \cite{GZ}. In this paper, we answer this question in the affirmative. 

We further generalize this by considering  subarrangements of the Type $B$ Coxeter arrangement. We call these arrangements $B$-graphical arrangements as they can be defined on symmetric graphs in a similar manner to graphical arrangements. In fact, the regions of a $B$-graphical arrangement $\mB_G$ are in bijection with symmetric acyclic orientations of the symmetric graph $G$. We generalize the concept of source components to these symmetric graphs and show that there are points $v \in \R^n$ such that for each region of $\mB_G$, the projection statistic given in~\cite{Kabluchko, LP} equals the number of signed source components of the symmetric acyclic orientation of $G$ labeling the region. Further, the characteristic polynomial of a $B$-graphical arrangement is equal to the chromatic polynomial of the associated symmetric graph, and hence we can generalize Greene and Zaslavky's result to symmetric graphs. In fact, symmetric graphs are equivalent to signed graphs, so we obtain a combinatorial interpretation for the coefficients of the chromatic polynomial of a signed graph. 

This paper is organized as follows. In Section~\ref{bgsec}, we set up some background and recall some results. In Section~\ref{braidsec}, we look at the braid arrangement $\mA_n = \mA_{K_n}$, and show that there is a particularly nice combinatorial interpretation for the projection statistic (see Theorem~\ref{braidRLmin}). In Section~\ref{gasec}, we find a set of points in $\R^n$ such that for each region of $\mA_G$, the projection statistic equals the number of source components of the acyclic orientation of $G$ labeling the region. This shows that the Greene and Zaslavsky result can be obtained using the Lofano and Paolini statistic. In Section~\ref{nuisec}, we look at graphical arrangements of natural unit interval graphs and show that the combinatorial statistic for the braid arrangement generalizes to these arrangements (see Theorem~\ref{NUIProjRL}). We use this interpretation to give an alternative proof of the form of the chromatic polynomial of a natural unit interval graph (see Corollary~\ref{NUIChar}). Finally, in Section~\ref{TypeBsec}, we generalize our results from Section~\ref{gasec} to subarrangements of the Type $B$ Coxeter arrangement. 

\section{Background}\label{bgsec}
In this section we state some definitions and recall some known results. 
We assume a standard background on graphs and hyperplane arrangements as given in \cite{Background}. We use \textit{graph} to mean an undirected finite graph without loops or multiple edges. For the rest of this section, let $G = ([n], E)$ be a graph, where $[n] = \{1, \ldots, n\}$. 

\begin{definition}\label{ChromaticPolynomial}
    Let $G = ([n], E)$ be a graph. The \emph{chromatic polynomial} of $G$, denoted by $\chi_G$, is the polynomial which when evaluated at a non-negative integer $q$ gives the number of proper $q$-colorings of $G$. 
\end{definition}

We now define \textit{source components} as in \cite{BN}. 
\begin{definition}\label{SourceComponents}
     Let $\ga$ be an acyclic orientation of $G$. For $i\in [n]$, let $R_i$ be the set of vertices reachable from $i$ by a directed path of $\ga$ (with $i \in R_i$). We define $S_1, S_2, \ldots$ recursively: for $k \geq 1$, if $\bigcup_{i < k} S_i = [n]$, then $S_k = \emp$. Otherwise, define $S_k = R_m \setminus \bigcup_{i < k} S_i$ where $m = \min\{[n] \setminus \bigcup_{i < k} S_i\}$.  The non-empty subsets $S_k$ thus defined are the \emph{source components} of $\gamma$.
\end{definition}

\begin{figure}[ht]
    \centering
    \includegraphics[width=0.75\linewidth]{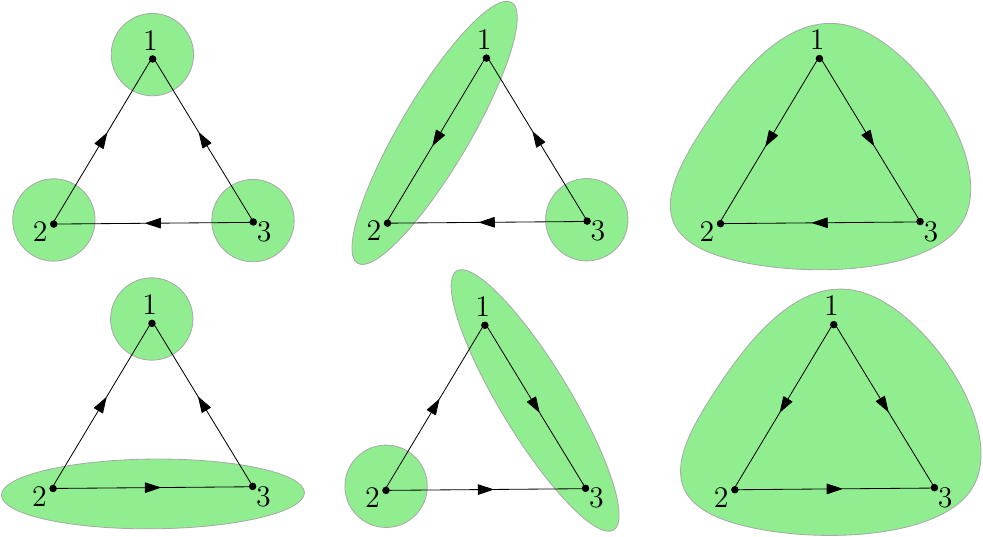}
    \caption{The source components of the 6 acyclic orientations of $K_3$.}
    \label{Fig:SourceComponents}
\end{figure}

The following result was proven by Greene and Zaslavsky. 
\begin{theorem}\cite{GZ} \label{GZProj}
    Let $G = ([n], E)$ be a graph, and $k$ be a non-negative integer. Then, 
    $$[q^k]\chi_G(q) = (-1)^{n-k}\alpha_k$$ where $\alpha_k$ is the number of acyclic orientations of $G$ with exactly $k$ source components. 
\end{theorem}

The characteristic polynomial of a hyperplane arrangement $\mA$, denoted by $\chi_{\mA}$, is classically defined using the Mobius function of the intersection lattice $\mathcal{L}_{\mA}$ as follows: $\chi_{\mA}(t) := \sli_{L \in \mathcal{L}_{\mA}} \mu(0, L)t^{\dim L}$. For our part, we use a characterization of the characteristic polynomial given by Lofano and Paolini~\cite{LP} and Kabluchko~\cite{Kabluchko}, which we now recall. 

\begin{definition}
    Let $\mA$ be a hyperplane arrangement in $\R^n$ and let $v \in \R^n$ be an arbitrary point. Let $R$ be a region of $\mA$. We define the \emph{projection from $v$ to $R$} to be the unique point in $R$ that has the minimum Euclidean distance from $v$, and denote it by \emph{$\textup{proj}_v(R)$}. Further, we define the \emph{projection dimension of $v$ on $R$} to be the dimension of the unique face of $R$ that contains $\proj_v(R)$ in its relative interior, and denote it by $\pd_v(R)$.
\end{definition}

\begin{figure}[ht]
    \centering
    \includegraphics[width=0.6\linewidth]{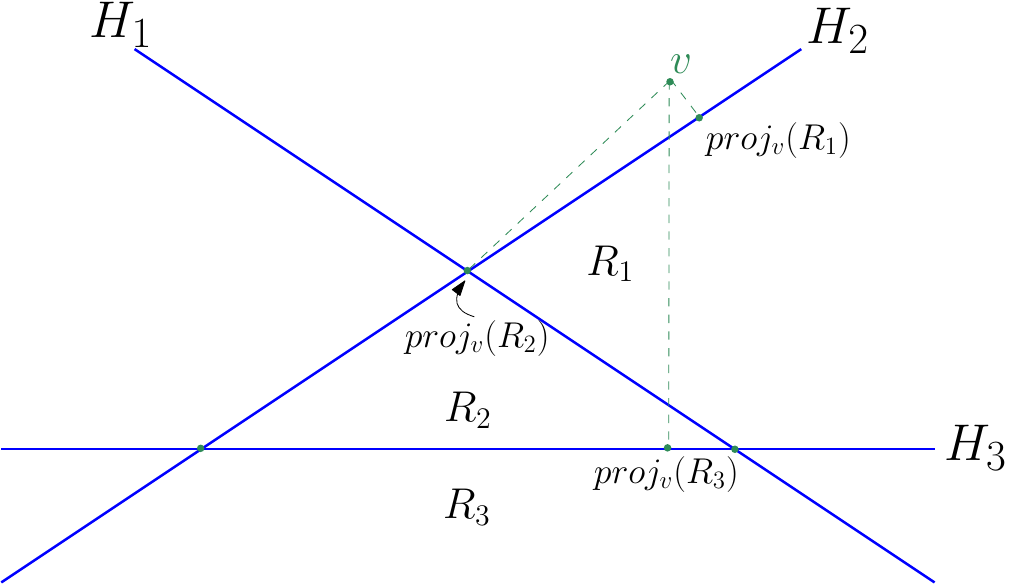}
    \caption{A hyperplane arrangement in $\R^2$ with a point $v \in \R^2$ and the projection from $v$ to some regions. Here, $\pd_v(R_1) = 1$, $\pd_v(R_2) = 0$, and $\pd_v(R_3) = 1$. }
    \label{fig:ProjStat}
\end{figure}

\begin{theorem}\cite{Kabluchko, LP}\label{Kab}
    Let $\mA$ be a hyperplane arrangement in $\R^n$, and let $v \in \R^n$ be a generic\footnote{For a detailed definition of generic, see~\cite{Kabluchko}. For our purposes, as long as the set of points we consider contains an open ball, it contains generic points.} point. Let $R(\mA)$ be the set of regions of $\mA$. 
    Then, the characteristic polynomial of $\mA$ is given by 
    $$\chi_{\mA}(t) = \sum\limits_{R \in R(\mA)} (-1)^{n - \pd_v(R)}t^{\pd_v(R)}.$$

    Equivalently, for a non-negative integer $k$, the coefficient of $t^k$ in $(-1)^n\chi_{\mA}(-t)$ is the number of regions $R$ of $\mA$ such that the projection dimension of $v$ on $R$ is $k$.
\end{theorem}

\section{The Braid Arrangement}\label{braidsec}
In this section, we show that for the braid arrangement $\mA_n$, the projection statistic defined in Theorem~\ref{Kab} can be expressed as a natural combinatorial statistic on the permutations labeling the regions of $\mA_n$.

\begin{definition}
    The \textit{braid arrangement} $\mA_n$ is defined as the collection of the following hyperplanes:
    $$\mA_n = \{H_{i,j} \mid 1 \leq i < j \leq n\}\,,$$
    where $H_{i,j} = \{(x_1, \ldots, x_n) \in \R^n \mid x_i - x_j = 0\}$.
\end{definition}

It is easy to see that each region of the braid arrangement is uniquely determined by a total ordering of the coordinates. As a consequence, there is a bijection between the regions of the braid arrangement $\mA_n$ and permutations of $[n]$. When considered in one line notation, the permutation indicates the relative order of the coordinates. For example, the permutation $\sigma = \sigma_1\sigma_2\ldots\sigma_n$ is associated to the region consisting of all points with coordinates $x_{\sigma_1} > x_{\sigma_2} > \ldots > x_{\sigma_n}$.

\begin{figure}[ht]
    \centering
    \includegraphics[width=0.8\linewidth]{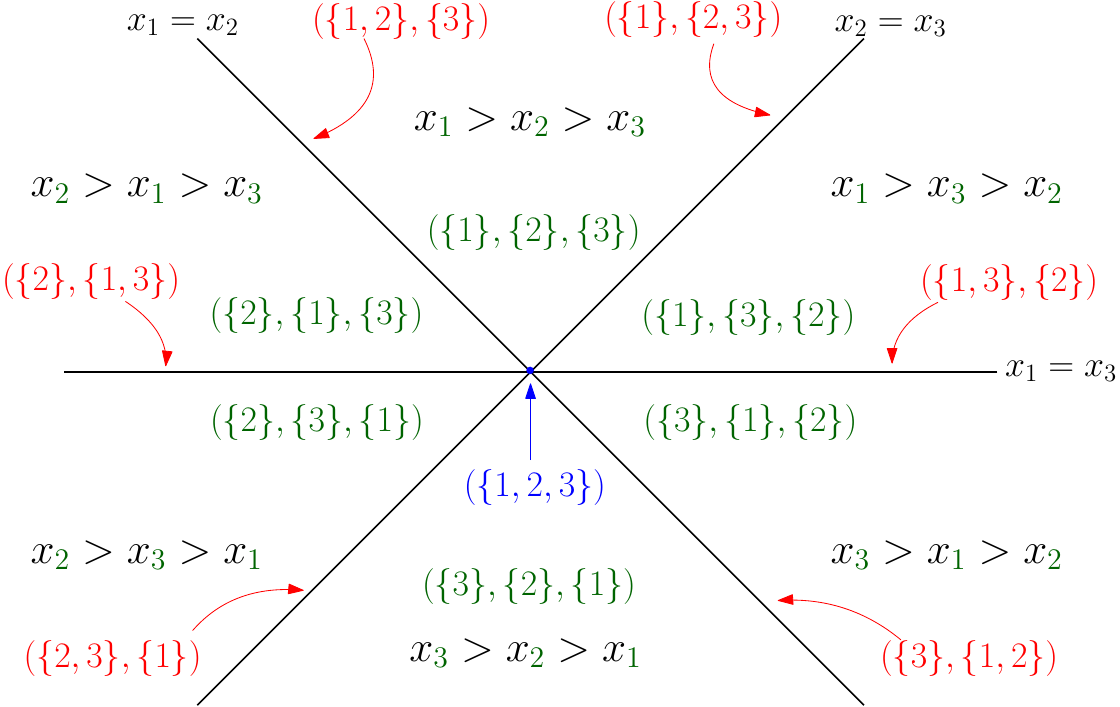}
    \caption[]{$\mathcal{A}_3$ with the faces labeled by ordered partitions of $[3]$.\footnotemark }
    \label{fig:BraidFaces}
\end{figure}
\footnotetext{Note that the hyperplane $H = \{(x_1, x_2, x_3) = x_1 + x_2 + x_3)\}$ is orthogonal to all the hyperplanes of $\mA_3$, so we represent $\mA_3$ by drawing its intersection with $H$.}

Moreover, the flats of the braid arrangement can be represented by partitions of $[n]$, the blocks indicating which coordinates are equal (that is, if $i$ and $j$ are in the same block, then $x_i = x_j$ for any point $x$ of the flat). Finally, the faces of the arrangement can be represented by ordered partitions of $[n]$, the relative order of the blocks indicating the relative order of the corresponding coordinates (that is, if $i$ is in a block before $j$, then $x_i > x_j$ for all points $x$ of the face). Note that the number of blocks of the partition is the dimension of the corresponding face. Given a face $F$ of $\mA$, we denote by $\Pi(F)$ the ordered partition labeling it, and given an ordered partition $\Pi$ of $[n]$, we denote by $F_{\Pi}$ the face it labels. 

We now state the main result of this section. 
\begin{definition}\label{RLmin}
    Let $\sigma = \sigma_1\ldots\sigma_n \in \Sn_n$ be a permutation. We say $\sigma_i$ is a \emph{right-to-left minimum} of $\sigma$ if $\sigma_i < \sigma_j$ for all $j > i$. We denote by $\text{RLmin}(\sigma)$ the number of right-to-left minima of $\sigma$. 
\end{definition}

\begin{theorem}\label{braidRLmin}
    Let $v = (v_1, \ldots, v_n) \in \R^n$ such that 
    \begin{equation}\label{vector}
        v_1 > \ldots > v_n \text{ and } v_i - v_{i+1} > n(v_{i+1} - v_n) \text{ for all } i < n.
    \end{equation}
    Let $\sigma \in \Sn_n$ be a permutation and let $R_{\sigma}$ be the region of the braid arrangement $\mA_n$ labeled by $\sigma$. Then the projection dimension of $v$ on $R_{\sigma}$ equals the number of right-to-left minima of $\sigma$, that is, $\pd_v(R_{\sigma}) = \textup{RLmin}(\sigma)$. 
\end{theorem}

\begin{figure}[ht]
    \centering
    \includegraphics[width=0.8\linewidth]{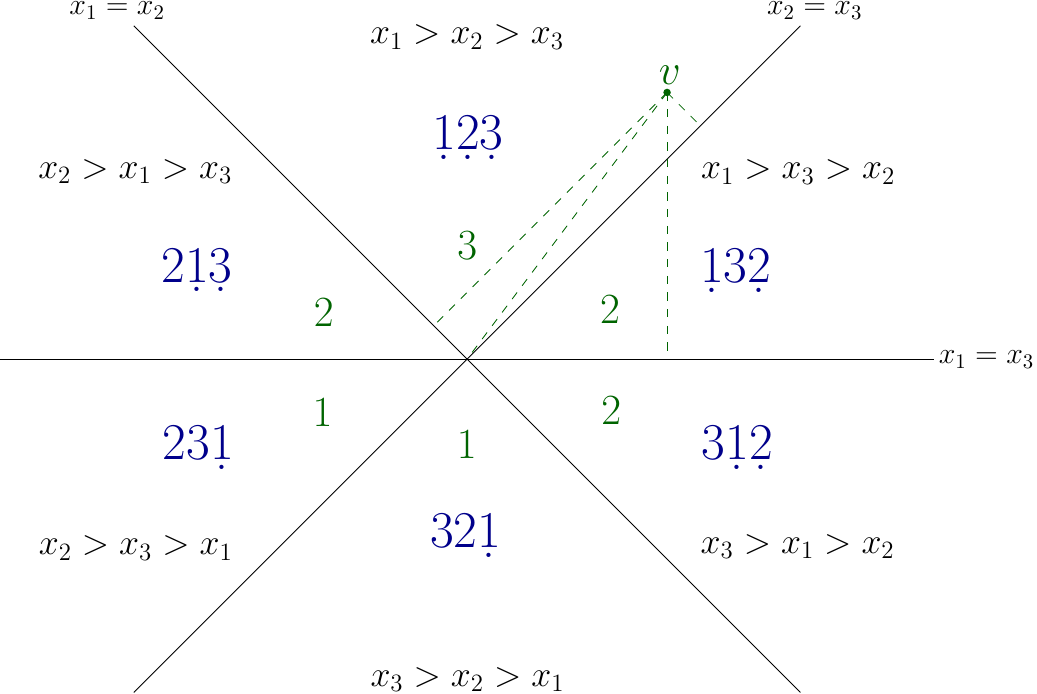}
    \caption{$\mB_3$ with projection from a point $v$ as in Equation~\eqref{vector} onto the regions, with the regions labeled by permutations (in blue) with right-to-left minima marked by dots, and dimension of projection/number of right-to-left minima (in green).}
    \label{fig:BraidProj}
\end{figure}

\begin{remark}
    Note that the set defined by Equation~\eqref{vector} is open and further, $v \in \R^n$ such that $v_i = (n+1)^{-i}$ for all $i \in [n]$ belongs to the set. Hence the set of $v \in \R^n$ for which the above theorem holds must contain generic points. 
\end{remark}

Observe that Theorem~\ref{braidRLmin} (along with Theorem~\ref{Kab} from ~\cite{Kabluchko, LP}) imply the following classical expression for the characteristic polynomial of the braid arrangement
\begin{equation*}
    \chi_{\mA_n}(t) = \sli_{\sigma\in \Sn_n} (-1)^{n}(-t)^{\text{RLmin}(\sigma)}.
\end{equation*}

The rest of this section is dedicated to the proof of Theorem~\ref{braidRLmin}.
 
Let $v \in \R^n$ be as in Equation~\eqref{vector}, and $R$ be a region of $\mA_n$. There is a unique face of $R$ that contains $\proj_v(R)$ in its interior. In order to prove the above theorem, we first prove some lemmas which allow us to characterize this face. 
\begin{lemma}\label{projchar}
    Let $v \in \R^n$ and let $\Pi = \{B_1, \ldots, B_k\}$ be an unordered set partition of $[n]$. If $i \in B_j = \{j_1, \ldots, j_s\}$, then the projection of $v$ on the flat labeled by $\Pi$ has the $i^{\text{th}}$ coordinate given by $\dfrac{v_{j_1} + \ldots + v_{j_s}}{s}.$
\end{lemma}
\begin{proof}
    Let $X$ be the flat of $\mA_n$ labeled by $\Pi$. Note that $\proj_v(X)$ is the orthogonal projection of $v$ on $X$.  

    For $j \in [k]$, define $f_j = \sli_{i \in B_j} e_i$ where $e_i$ is the $i^{\text{th}}$ standard basis vector of $\R^n$. Note that the set $O = \{f_j\}_{j \in  [k]}$ forms an orthogonal basis of $X$. Hence, $\proj_v(X) = \sli_{j \in [k]} \dfrac{\langle v, f_j \rangle}{||f_j||^2}f_j$. 

    Now, if $i \in B_j$, the $i^{\text{th}}$ coordinate of $f_t$ for $t \neq j$ is $0$, and the $i^{\text{th}}$ coordinate of $f_j$ is $1$. Hence, the $i^{\text{th}}$ coordinate of $\proj_v(X)$ is $\dfrac{\langle v, f_j\rangle }{||f_j||^2} = \dfrac{v_{j_1} + \ldots + v_{j_s}}{s}$. 
\end{proof}

\begin{definition}\label{goodface}
    Let $\mA$ be a hyperplane arrangement in $\R^n$, and let $v \in \R^n$ be a point. Let $F$ be a face of $\mA$. We say $F$ is a \emph{$v$-face} if the projection of $v$ on span($F$) lies in the relative interior of $F$. Further, for $R$ a region of $\mA$, we say a face $F$ is a \emph{$v$-face of $R$} if $F$ is a $v$-face and $F$ is incident to $R$. 
\end{definition} 

The following lemma characterizes the $v$-faces of $\mA_n$.
\begin{lemma}\label{GFCond}
    Let $v \in \R^n$ be as in Equation~\eqref{vector}, and let $\Pi = (B_1, \ldots, B_k)$ be an ordered partition of $[n]$. Then $F_{\Pi}$ is a $v$-face of $\mA_n$ if and only if $\min(B_i) < \min(B_{i+1})$ for all $i < k$. 
\end{lemma}
\begin{proof} 
    Let $B = \{i_1, \ldots, i_k\}$ and $B' = \{i_1', \ldots, i_{\ell}'\}$ be two blocks of our partition, where we write their elements in increasing order. Then $\min(B) = i_1$ and $\min(B') = i_1'$. Without loss of generality, let $i_1 < i_1'$. 

    Let $X = \text{span}(F_{\Pi})$, that is, $X$ is the flat labeled by the unordered partition $\{B_1, \ldots, B_k\}$. 
    Let $\proj_v(X) = (p_1, \ldots, p_n)$. Note that as $\proj_v(X)$ lies in a unique face of $X$, it is enough to show that the ordered partition labeling this face satisfies $\min(B_i) < \min(B_{i+1})$ for all $i < k$. 

    Then, from Lemma \ref{projchar}, and the fact that $v_{i_1} - v_{i_1' -1} + v_{i_2} - v_n + \ldots + v_{i_k} - v_n > 0$, we have: 
    $$p_{i_1} = \dfrac{v_{i_1} + \ldots + v_{i_k}}{k} \geq \dfrac{(k-1)v_n + v_{i_1' - 1}}{k},$$
    and, as $i_1' \leq i_j'$ for all $j \in [k]$, as $v$ satisfies Equation~\eqref{vector}, we have $v_{i_j'} \leq v_{i_1'}$, and hence,
    $$p_{i_1'} = \dfrac{v_{i_1'} + \ldots + v_{i_{\ell'}}}{\ell} \leq v_{i_1'}.$$
    This gives us 
    \begin{align*}
        p_{i_1} - p_{i_1'} \geq \dfrac{v_{i_1' - 1} - v_{i_1'} - (k-1)(v_{i_1'} - v_n)}{k} 
        > \dfrac{v_{i_1' - 1} - v_{i_1'} - n(v_{i_1'} - v_n)}{k}
        >0.
    \end{align*}

    Hence $B$ must appear before $B'$ in the ordered partition.  
\end{proof}

\begin{lemma}\label{ProjOnGoodFace}
    Let $\mA$ be a hyperplane arrangement in $\R^n$, let $v \in \R^n$ be a point and let $R$ be a region of $\mA$.
    Then, $$\proj_v(R) = \proj_v(F)$$ where $F$ is a $v$-face of $\mA$ incident to $R$.
\end{lemma}
\begin{proof}
    Clearly, there is a unique face $F$ of $R$ that contains $\proj_v(R)$ in its relative interior. Hence, $\proj_v(R) = \proj_v(F)$, and we need to show that $F$ is a $v$-face. 

    Suppose $F$ is not a $v$-face. Then, by the definition of a $v$-face, for $X = \s(F)$, we have $\proj_v(X) \neq \proj_v(F)$. Let $L$ be the line joining $\proj_v(X)$ and $\proj_v(F)$. As $\proj_v(X)$ is the orthogonal projection of $v$ onto $X$ and $L$ lies in $X$, the distance between $v$ and a point $p \in L$ increases as $p$ moves from $\proj_v(X)$ to $\proj_v(F)$. But then $\proj_v(F)$ must be on the boundary of $F$ as otherwise we would have points on $L$ in $F$ closer to $v$, which contradicts the fact that $\proj_v(F)$ minimizes distance. This contradicts the fact that $\proj_v(F)$ is in the relative interior of $F$, so $F$ must be a $v$-face.
\end{proof}

Now, we finally have enough information to prove Theorem~\ref{braidRLmin}.
\begin{proof}[Proof of Theorem~\ref{braidRLmin}]
    Let $\sigma = \sigma_1\ldots \sigma_n$. For any face of $R_{\sigma}$, the ordered partition labeling it is of the form $(\{\sigma_1, \ldots, \sigma_{i_1}\}, \{\sigma_{i_1 + 1}, \ldots, \sigma_{i_2}\}, \ldots, \{\sigma_{i_k + 1}, \ldots, \sigma_{n}\})$ for some $k \geq 0$, and $0 < i_1 < \ldots < i_{k}< n$. 

    Now, by Lemma~\ref{ProjOnGoodFace}, $\proj_v(R_{\sigma})$ must lie on a $v$-face of $R_{\sigma}$. Lemma~\ref{GFCond} further gives us that these $v$-faces are precisely those labeled by ordered partitions of $[n]$ of the form $\Pi = (B_1, \ldots, B_k)$ such that $\min(B_i) < \min(B_{i+1})$ for all $i < k$.

    This implies that if $F_{\Pi}$ is a $v$-face of $R_{\sigma}$, then each block $B_i$ contains at least one right-to-left minimum. If not, suppose $j$ is the greatest index such that $B_j$ does not contain a right-to-left minimum. Then, as $\sigma_n $ is always a right-to-left minimum, $j<k$ and further $\min(B_j) > \min(B_{j+1})$, contradicting that $F_{\Pi}$ is a $v$-face. Hence, the dimension of a $v$-face cannot be greater than \text{RLmin($\sigma$)}.

    Now, suppose that $F_{\sigma}$ is the face of $R_{\sigma}$ that $v$ projects into. If the dimension of $F_{\sigma}$ is less than RLmin($\sigma$), there is a block of $\Pi(F_{\sigma})$ that has more than one right-to-left minimum. We can refine the partition (by breaking this block into two blocks, each containing at least one right-to-left minimum) to get a partition corresponding to a larger face $F'$ which contains $F_\sigma$. This gives us the projection from $v$ onto $F'$ must be of shorter length (strictly shorter as $F'$ is a $v$-face) than the projection onto $F_{\sigma}$, which is a contradiction. 

    Hence we have $\pd_v(R_{\sigma}) = \text{RLmin($\sigma$)}$.
\end{proof}

\begin{remark}
    For a point $v = (v_1, \ldots, v_n) \in\R^n$ such that $v_1 > \ldots > v_n$ and $v_i - v_{i+1} > n(v_1 - v_i)$ for all $i < n$, one can similarly show that for a permutation $\sigma = \sigma_1\ldots\sigma_n \in \Sn_n$, $$\pd_v(R_{\sigma}) = \text{LRmax}(\sigma) = \# \text{ of left-to-right maxima of } \sigma, $$ where $\sigma_i$ is a left-to-right maximum of $\sigma$ if $\sigma_i > \sigma_j$ for all $j<i$. 
\end{remark}

\section{Graphical Arrangements}\label{gasec}
In this section, we show that for a graphical arrangement, the projection statistic corresponds to Greene and Zaslavsky's statistic from~\cite{GZ}. We also show that this is a generalization of our results from Section~\ref{braidsec}. 

\begin{definition}
    Let $G = ([n], E)$ be a graph. The \textit{graphical arrangement} $\mA_G$ is defined as the collection of the following hyperplanes:
    $$\mA_G := \{H_{i,j} \mid \{i, j\} \in E, \, i<j\}\,,$$
    where $H_{i,j} = \{(x_1, \ldots, x_n) \in \R^n \mid x_i - x_j = 0\}$.
\end{definition}

\begin{definition}
    Let $R \in R(\mA_G)$. We associate to $R$ an acyclic orientation of $G$, denoted by $\gamma_R$, by directing the edge $\{i,j\} \in E$ towards $i$ if and only if $x_i \geq x_j$ in $R$. 
\end{definition}

We now recall a classical fact. 

\begin{lemma}\cite{BG2}
    The mapping $\gamma$ which associates to each region $R \in R(\mA_G)$ the orientation $\gamma_R$ is a bijection between $R(\mA_G)$ and the set of acyclic orientations of $G$. 
\end{lemma}

\begin{figure}[ht]
    \centering
    \includegraphics[width=0.8\linewidth]{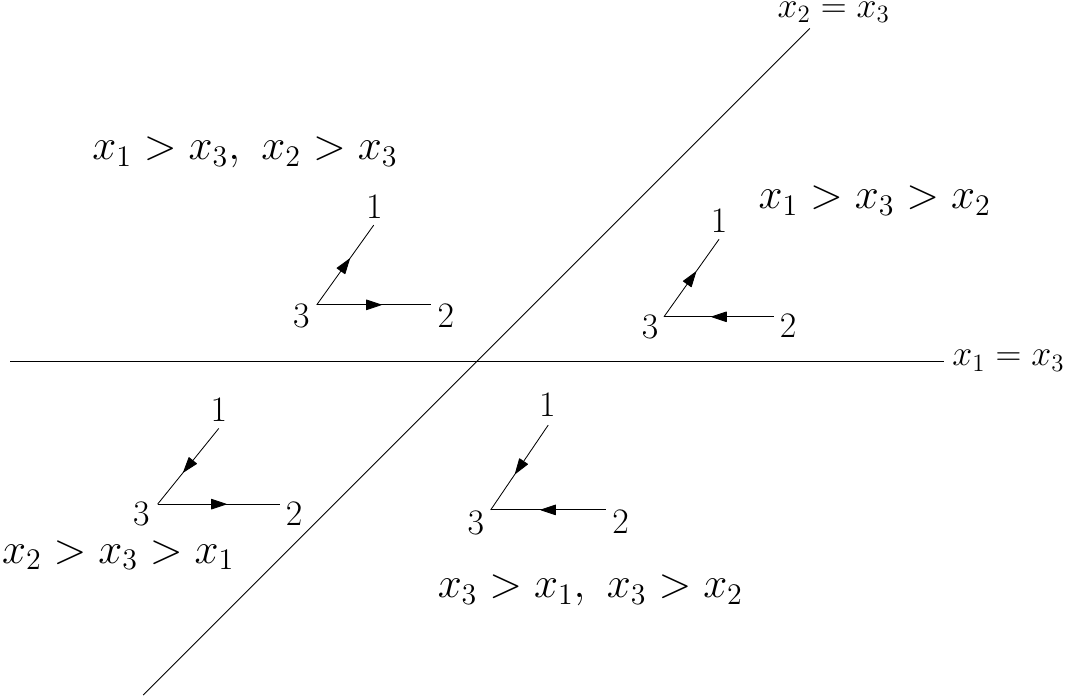}
    \caption{A graphical arrangement with the regions labeled by the acyclic orientations given by $\gamma$.}
    \label{fig:GraphicalLabeling}
\end{figure}

Recall that the faces of the braid arrangement $\mA_n$ can be identified with ordered partitions of $[n]$. As $\mA_G$ is a sub-arrangement of $\mA_n$, each flat of $\mA_G$ corresponds to a partition of $[n]$ and each face of $\mA_G$ corresponds to a set of ordered partitions of $[n]$. 

Note that for an acyclic orientation $\gamma$ of $G$, the source components $(S_1, \ldots, S_k)$ of $\gamma$ form an ordered partition of $[n]$. We denote this ordered partition by $\Pi(\gamma)$.

The main result of this section is the following:
\begin{theorem}\label{graphSC}
    Let $v = (v_1, \ldots, v_n) \in \R^n$ be such that 
    \begin{equation}\label{vec2}
        \forall i \in [n-1], \text{ }v_i > (6n^2 + 1)v_{i+1} \text{ and } v_n > 0. 
    \end{equation}
     Let $G = ([n],E)$ be a graph, let $R$ be a region of the graphical arrangement $\mA_G$, and let the acyclic orientation $\gamma_R$ have $k$ source components. 

     Then, $\pd_v(R) = k$, that is, the projection dimension of $v$ on $R$ equals the number of source components of $\gamma_R$. In fact, the face of $R$ that $\proj_v(R)$ lies in the relative interior of is $F_{\Pi(\gamma_R)}$.
\end{theorem}

\begin{figure}[ht]
    \centering
    \includegraphics[width=0.8\linewidth]{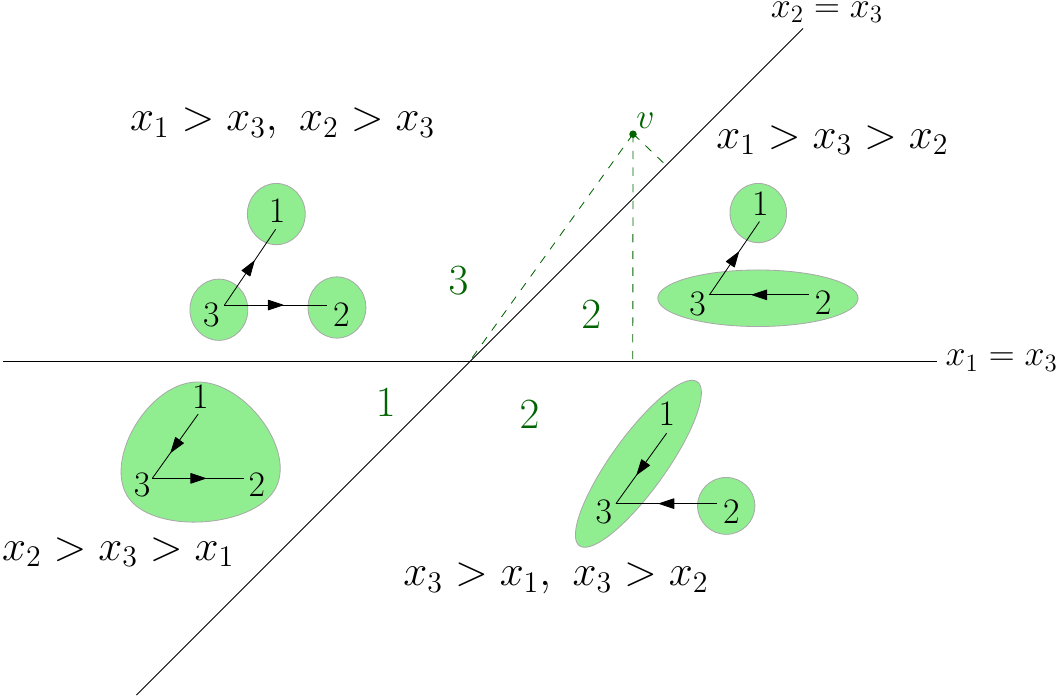}
    \caption{A graphical arrangement with projection from a point $v$ as in Equation \eqref{vec2} onto the regions, with the regions labeled by acyclic orientations with source components, and dimension of projection/number of source components in green.}
    \label{fig:GraphicalProjection}
\end{figure}

\begin{remark}
    Note that the set defined by Equation~\eqref{vec2} is open and further, $v \in \R^n$ such that $v_i = (6n^2 + 2)^{-i}$ for all $i \in [n]$ satisfies the conditions of Equation~\eqref{vec2}. Hence the set of $v\in\R^n$ for which the above theorem holds contains generic points. Further, any point satisfying the conditions of Equation~\eqref{vec2} will also satisfy the conditions of Equation~\eqref{vector}. Hence, any results that hold for a point $v$ as in Equation~\eqref{vector} will also hold for a point $v$ as in Equation~\eqref{vec2}.
\end{remark}

\begin{remark}
    In the case of the braid arrangement, Theorem~\ref{graphSC} along with the following lemma imply Theorem~\ref{braidRLmin}. 
    In fact, they imply a stronger result as they show that for a region $R_{\sigma}$ of the braid arrangement labeled by  $\sigma \in \Sn_n$, the face that $\proj_v(R_{\sigma})$ lies in the relative interior of is labeled by the ordered partition formed by partitioning $\sigma$ at its right-to-left minima. 
\end{remark}

\begin{lemma}\label{rlminsc}
    The number of right-to-left minima of a permutation $\sigma \in \Sn_n$ is equal to the number of source components of  the acyclic orientation $\gamma_{R_{\sigma}}$ of $K_n$, where $R_{\sigma}$ is the region of $\mA_n$ labeled by $\sigma$. 
\end{lemma}
\begin{proof}
    Let $\sigma \in \Sn_n$. Then, $i \in [n]$ is not a right-to-left minimum of $\sigma$ if and only if there is a $j < i$ such that $j$ appears after $i$ in the one line notation of $\sigma$. The only way this can happen is if $x_i \geq x_j$ for all  $x \in R_{\sigma}$. 

    Further, each source component of $\gamma_{R_{\sigma}}$ can be identified with the least element in that source component. Then $i \in [n]$ is not the least element in its source component if and only if there exists $j < i$ such that $i$ is $j$-reachable. Let the path from $j$ to $i$ be $jk_1\ldots k_r i$. Then, $x_j \leq x_{k_1} \leq \ldots \leq x_{k_r} \leq x_i$ for all $x \in R_{\sigma}$. Hence, $i \in [n]$ is not the least element in its source component if and only if there exists $j < i$ such that $x_j \leq x_i$ for all $x \in R_{\sigma}$. 

    Hence, $i \in [n]$ is not a right-to-left minimum of $\sigma$ if and only if it is not the least element in its source component. 
     
    So, the number of right-to-left minima of $\sigma \in \Sn_n$ is equal to the number of source components of the acyclic orientation $\gamma_{R_{\sigma}}$ of $K_n$. 
\end{proof}

Further, Theorem~\ref{graphSC} along with Theorem~\ref{Kab} and Equation~\eqref{CC} give us an alternative proof of Greene and Zaslavsky's result about the interpretation of the coefficients of the characteristic polynomial (Theorem~\ref{GZProj}).  

The rest of this section is devoted to the proof of Theorem~\ref{graphSC}.

\begin{lemma} \label{ubdim}
    Let $G = ([n], E)$ be a graph, let $R$ be a region of $\mA_G$, and let $\gamma_R$ be the acyclic orientation of $G$ labeling $R$. Let $\Pi(\gamma_R) = (B_1, \ldots, B_k)$. Then, for $v \in \R^n$ as in Equation~\eqref{vector}, $F_{\Pi(\gamma_R)}$ is a $v$-face of $R$ and $R$ does not have a $v$-face of dimension greater than $k$. 
\end{lemma}
\begin{proof}
    Let $\Pi$ be an ordered partition of $[n]$. Then, $\text{span}(F_{\Pi})$ is a flat of $\mA_G$ if and only if for every block $B$ of $\Pi$, the induced subgraph $G[B]$ is connected. This is because two coordinates $x_i$ and $x_j$ can be equated if and only if we have a set of hyperplanes $x_i = x_{k_1}$, $x_{k_1} = x_{k_2}$, \ldots, $x_{k_t} = x_j$, which corresponds to a path between $i$ and $j$ in $G$. 

    Further, $F_{\Pi}$ is a face of a region $R$ if and only if any weak inequality in $F_{\Pi}$ holds in $R$. Hence, it is easy to see that $F_{\Pi(\gamma_R)}$ is a face of $R$. Also, for $b_i = \min B_i$, we have $b_1 < b_2 < \ldots < b_k$, and hence $F_{\Pi(\gamma_R)}$ is a $v$-face of $R$ by Lemma~\ref{GFCond}.

    Now, let $F'$ be a $v$-face of $R$. This face corresponds to a set of ordered partitions of $[n]$. We consider $\Pi(F') = (D_1, \ldots, D_{\ell})$ to be the ordered partition in the set satisfying $\min(D_i) < \min(D_{i+1})$ for all $i < \ell$. Note that such a partition exists as $F'$ is a $v$-face. 

    \begin{adjustwidth}{2em}{0pt}
    \noindent   \textbf{Claim:}     
        \begin{equation}\label{containment}
            \text{For all  } j \in [k], \,\, \bigcup\limits_{i = 1}^{j} B_i \sse \bigcup\limits_{i = 1}^{j}D_i, 
        \end{equation}
        
    \noindent  where we use the convention $D_j = \emp$ for all $j > \ell$.
             
    \noindent We prove this using induction on $j \in [k]$. The base case for $j = 0$ is trivial. 

    \noindent Now, suppose we have $\bigcup\limits_{i = 1}^{j-1} B_i \sse \bigcup\limits_{i = 1}^{j-1}D_i$. Then two cases arise:

    \begin{adjustwidth}{2em}{0pt}
        \textbf{Case 1:} $b_j \in \bigcup\limits_{i = 1}^{j-1} D_i$.

    \noindent Let $m \in [n]$ be $b_j$-reachable. Then, $x_m \geq x_{b_j}$ for all $x \in R$. Now, if $m \in D_t$ for some $t \geq j$, we get $x_{b_j} > x_m$ for all $x \in R$, which is a contradiction. 

    \noindent Hence, $m \in \bigcup\limits_{i = 1}^{j-1} D_i$, which gives us $B_j \sse \bigcup\limits_{i = 1}^{j-1} D_i \sse \bigcup\limits_{i = 1}^{j} D_i$.

    \noindent \textbf{Case 2:} $b_j \notin \bigcup\limits_{i = 1}^{j-1} D_i$.

    \noindent Then $b_j \in D_j$, as if $b_j \in D_t$ for some $t > j$, we have $\min D_j > \min D_t$, contradicting the fact that $F'$ is a $v$-face by Lemma~\ref{GFCond}. Further, as $b_j \in D_j$, we have $B_j \sse \bigcup\limits_{i = 1}^{j}D_i$ using the same argument as above. 
    \end{adjustwidth}
    
    \noindent Hence our claim is true by induction on $j$. 
    \end{adjustwidth}

    As $\bigcup\limits_{i = 1}^k B_i = [n]$, we have $[n] \sse \bigcup\limits_{i = 1}^k D_i$. So, $\ell \leq k$. 

    Thus $R$ does not have a $v$-face of dimension greater than $k$. 
\end{proof}

\begin{lemma}\label{lbdim}
    Let $v$ be as in Equation~\eqref{vec2}. Let $R$ be a region of $\mA_G$, and let $\Pi(\gamma_R) = (B_1, \ldots, B_k)$. Let $p_B = \proj_v(F_{\Pi(\gamma_R)})$. Let $\Pi' = (D_1, \ldots, D_{\ell}) \neq (B_1, \ldots, B_k)$ be an ordered partition of $[n]$ such that $F_{\Pi'}$ is a $v$-face of $R$, and $p_D = \proj_v(F_{\Pi'})$.  

    Then, $||p - p_B|| < ||p - p_D||$. 
\end{lemma}
\begin{proof}
    From Lemma \ref{projchar}, $p_B = (p_1, \ldots, p_n)$, where $\forall i \in B_j$, $p_i = \sli_{m \in B_j} \dfrac{v_m}{|B_j|}$ and $p_D = (q_1, \ldots, q_n)$, where $\forall i \in D_j$, $q_i = \sli_{m \in D_j} \dfrac{v_m}{|D_j|}$.

    Now, from Equation~\eqref{containment}, $B_1 = D_1$ or $B_1 \subsetneq D_1$, and if $B_i = D_i$ for all $i < j$, $B_j = D_j$ or $B_j \subsetneq D_j$. Further, from Lemma~\ref{ubdim} we have $\ell \leq k$, and as $(D_1, \ldots, D_{\ell}) \neq (B_1, \ldots, B_k)$ we must have $D_j \supsetneq B_j$ for some $j$. 

    Let $j$ be the first index where $D_j \supsetneq B_j$ and let $|D_j| = d, |B_j| = b$.

    Then, 
    \begin{align*}
        ||v - p_B||^2 &= \sli_{i = 1}^n (v_i - p_i)^2 \\
        &=  \sli_{i \in B_m, m < j} (v_i - p_i)^2 + \sli_{i \in B_j}(v_i - p_i)^2 + \sli_{i \in B_m, m > j} (v_i - p_i)^2.
    \end{align*}
    Now,
    \begin{align*}
        (v_{b_j} - p_{b_j})^2 = \left(\dfrac{(b-1)}{b}v_{b_j} - \dfrac{1}{b}\sli_{\substack{i \in B_j \\ i \neq b_j}} v_i\right)^2 = \left(\dfrac{b-1}{b}\right)^2v_{b_j}^2 + \eps_1,
    \end{align*}
    where
    \begin{align*}
        \eps_1 = \left(\dfrac{1}{b}\sli_{\substack{i \in B_j \\ i \neq b_j}}v_i\right)^2 - \dfrac{2(b-1)}{b}v_{b_j}\left(\dfrac{1}{b}\sli_{\substack{i \in B_j \\ i \neq b_j}}v_i\right). 
    \end{align*}
    Now, as $v_i \geq (6n^2 + 1)v_{i+1}$ for all $i < n$,
    \begin{align*}
        \frac{1}{b}\sli_{\substack{i \in B_j \\ i \neq b_j}} v_i &\leq \frac{1}{b}\left(\frac{1}{6n^2 + 1}v_{b_j} + \frac{1}{(6n^2 + 1)^2}v_{b_j} + \ldots + \frac{1}{(6n^2 + 1)^{b-1}}v_{b_j} \right) \\
        &\leq \frac{1}{6n^2b}v_{b_j}.
    \end{align*}
    Hence,
    \begin{align*}
        |\eps_1| \leq \max\left\{\frac{2(b-1)}{6n^2b^2}v_{b_j}^2, \frac{1}{36n^4b^2}v_{b_j}^2\right\} = \frac{2(b-1)}{6n^2b^2}v_{b_j}^2.
    \end{align*}

    Further, 
    \begin{align*}
        \sli_{\substack{i \in B_j \\ i \neq b_j } }(v_i - p_i)^2 = \sli_{\substack{i \in B_j \\ i \neq b_j}}\left(\dfrac{v_{b_j}}{b} + \left(\sli_{\substack{\ell \in B_j \\ \ell \neq b_j}}\dfrac{v_{\ell}}{b} - v_i\right)\right)^2 = \dfrac{(b-1)}{b^2}v_{b_j}^2 + \eps_2,
    \end{align*}
    where
    \begin{align*}
        \eps_2 = \sli_{\substack{i \in B_j \\ i \neq b_j}}\dfrac{2v_{b_j}}{b}\left(\sli_{\substack{\ell \in B_j \\ \ell \neq b_j}}\dfrac{v_{\ell}}{b} - v_{i}\right)
        + \sli_{\substack{i \in B_j \\ i \neq b_j}}\left(\sli_{\substack{\ell \in B_j \\ \ell \neq b_j}}\dfrac{v_{\ell}}{b} - v_{i}\right)^2. 
    \end{align*}
    
    Now, as $v_i \geq (6n^2 + 1)v_{i+1}$ for all $i < n$,
    \begin{align*}
        \left|\sli_{\substack{\ell \in B_j \\ \ell \neq b_j}}\frac{v_{\ell}}{b} - v_i\right| \leq \frac{1}{6n^2 + 1}v_{b_j},
    \end{align*}
    and hence,
    \begin{align*}
        |\eps_2| \leq \frac{2(b-1)}{b(6n^2 + 1)}v_{b_j}^2 + \frac{(b-1)}{(6n^2 + 1)^2}v_{b_j}^2.
    \end{align*}
    Finally, for $i \in B_m$, $m > j$,
    \begin{align*}
        (v_i - p_i) \leq \frac{1}{6n^2 + 1}v_{b_j},
    \end{align*}
    and hence, 
    \begin{align*}
        |\eps_3| = \left|\sli_{i \in B_m, m > j} (v_i - p_i)^2\right| \leq \frac{2(n-b)}{(6n^2 + 1)^2}v_{b_j}^2.
    \end{align*}
    Now, for $\eps = \eps_1 + \eps_2 + \eps_3$,
    \begin{align*}
        |\eps| &\leq |\eps_1| + |\eps_2| + |\eps_3| \\
        &\leq \left(\frac{2(b-1)}{6n^2b^2} + \frac{2(b-1)}{b(6n^2 + 1)} + \frac{(b-1)}{(6n^2 + 1)^2} + \frac{2(n-b)}{(6n^2 + 1)^2}\right)v_{b_j}^2 \\
        &\leq \frac{1}{6n^2}\left(\frac{2(b-1)}{b^2} + \frac{2(b-1)}{b} + \frac{2n - b - 1}{6n^2 + 1}\right)v_{b_j}^2 \\
        &\leq \frac{1}{2n^2}v_{b_j}^2,
    \end{align*}
    and hence,
    \begin{align*}
        ||v - p_B||^2 
        \leq \sli_{i \in B_m, m < j} (v_i - p_i)^2 + \left(1 - \dfrac{1}{b}\right)v_{b_j}^2 + \dfrac{1}{2n^2}v_{b_j}^2.
    \end{align*}
    In a similar manner,  
    \begin{align*}
        ||v - p_D||^2 
        \geq \sli_{i \in D_m, m < j} (v_i - p_i)^2 + \left( 1- \dfrac{1}{d}\right)v_{b_j}^2 - \dfrac{1}{2n^2}v_{b_j}^2.
    \end{align*}
    
    Now, as $B_i = D_i$ for all $i < j$, and $d > b$, 
    \begin{align*}
        ||v - p_D||^2 - || v - p_B||^2 \geq \left(\dfrac{1}{b} - \dfrac{1}{d} - \dfrac{1}{n^2}\right)v_{b_j}^2 > 0. 
    \end{align*}
\end{proof}
Theorem~\ref{graphSC} is now a direct consequence of Lemmas~\ref{ubdim} and~\ref{lbdim}.

\section{Natural Unit Interval Graphs}\label{nuisec}
In this section, we consider a special type of graphs known as natural unit interval graphs. We show that for the graphical arrangements of these graphs, the projection statistic is a generalization of the $\text{RLmin}$ statistic for the braid arrangement. We further show that this can be used to give the general form of the characteristic polynomial of such a graphical arrangment. 

\begin{definition}
    A graph $G = ([n], E)$ is a \emph{natural unit interval graph} if for all $\{i,j\} \in E$, with $i<j$, we have $\{i,k\}\in E$ and $\{k,j\} \in E$ for all $i < k < j$. 
\end{definition}

We now consider the graphical arrangement $\mA_G$ of a natural unit interval graph $G$. We know that every region of the braid arrangement is indexed by a permutation. As a region of a graphical arrangement is a union of 
adjacent regions of the braid arrangement, each region $R$ of our graphical arrangement $\mA_G$ is uniquely associated to a set of permutations. Let us denote this set by $S_R$. Further, note that if $\sigma = \sigma_1 \ldots \sigma_n \in S_R$, then $\sigma' = \sigma_1 \ldots \sigma_{i + 1} \sigma_i\ldots \sigma_n \in S_R$ if and only if $x_{\sigma_i} - x_{\sigma_{i+1}} = 0$ is not a hyperplane of the arrangement.

We now state the main results of this section:
\begin{theorem}\label{NUIProjRL}
    Let $G = ([n],E)$ be a natural unit interval graph, and let $R$ be a region of $\mA_G$. Let $\sigma\in \Sn_n$ be the lexicographic minimum permutation of $S_R$. Let $v$ be as in Equation~\eqref{vec2}. Then, 
    $$\pd_v(R) = \emph{RLmin}(\sigma).$$
\end{theorem}

\begin{corollary}\label{NUIChar}
    Let $G = ([n], E)$ be a natural unit interval graph. Then, 
    $$\chi_{\mA_G}(q) = \prod\limits_{j = 1}^n (q - c_j),$$
    where $c_j = |\{i < j \mid \{i, j\} \in E\}|$ for all $j \in [n]$.
\end{corollary}

Note that Corollary~\ref{NUIChar} is known for the chromatic polynomial of a natural unit interval graph (see, for instance, \cite{NUI}). Hence, Corollary~\ref{NUIChar} provides a new interpretation of this classical result. The rest of this section is dedicated to the proof of Theorem~\ref{NUIProjRL} and Corollary~\ref{NUIChar}. 

\begin{definition}
    Let $G = ([n], E)$ be a natural unit interval graph. A permutation $\sigma = \sigma_1\ldots \sigma_n \in \Sn_n$ is said to be a \emph{$G$-local minimum} if for all $i \in [n-1]$ such that $\sigma_i > \sigma_{i+1}$, we have $\{\sigma_i, \sigma_{i+1}\} \in E$. 
\end{definition}

We now characterize $G$-local minima. To do this, we first define $G$-descents. 
\begin{definition}
    Let $G = ([n], E)$ be a natural unit interval graph and $\sigma = \sigma_1\ldots \sigma_n \in \Sn_n$. For $i \in [n-1]$ we say we have a \emph{$G$-descent} at $i$ (or that $\sigma_i\sigma_{i+i}$ is a \emph{$G$-descent}) if $\sigma_i > \sigma_{i+1}$ and $\{\sigma_i, \sigma_{i+1}\} \in E$. 
\end{definition}

Note that a permutation $\sigma \in \Sn_n$ is a $G$-local minimum if and only if every descent of $\sigma$ is a $G$-descent. 

We further have:
\begin{lemma}\label{locallex}
    Let $G = ([n], E)$ be a natural unit interval graph. Let $R$ be a region of $\mA_G$ and let $S_R$ be the set of permutations associated to $R$. Then there is a unique $G$-local minimum in $S_R$ and it is in fact the lexicographic minimum of $S_R$. 
\end{lemma}
\begin{proof}
    Let $\sigma = \sigma_1\ldots \sigma_n$ be a $G$-local minimum in $S_R$ and $\tau = \tau_1 \ldots \tau_n$ be the lexicographic minimum of $S_R$. It is clear that $\tau$ is a $G$-local minimum as if there were a descent $\tau_i > \tau_{i+1}$ of $\tau$ such that $\{\tau_i, \tau_{i+1}\} \notin E$, we could swap $\tau_i$ and $\tau_{i+1}$ and get a permutation in $S_R$ that is less than $\tau$ in the lexicographic order. 
    
    We assume for contradiction that $\sigma \neq \tau$. Suppose they first differ at the $i^{\text{th}}$ position, that is, $\sigma_1 \ldots \sigma_{i-1} = \tau_1 \ldots \tau_{i-1}$ and $\sigma_i \neq \tau_i$. As $\tau$ and $\sigma$ are permutations, $\tau_i = \sigma_k$ for some $k > i$. 

    Now, as both $\sigma$ and $\tau$ correspond to the same region of $\mA_G$, we can obtain one from the other by a sequence of swapping adjacent elements in the one line notation, where the pairs of swapped elements are of the form $\{a,b\}$ with $\{a,b\} \notin E$. 

    We have $\sigma = \tau_1\ldots \tau_{i-1} \sigma_{i}\ldots \sigma_{k-1}\tau_i \sigma_{k+1} \ldots \sigma_n$. To obtain $\tau$ from $\sigma$, we would have to swap $\tau_i$ with $\sigma_t$ for all $i \leq t \leq k-1$. Hence, for all $i \leq t \leq k-1$, $x_{\tau_i} = x_{\sigma_t}$ is not a hyperplane of $\mA_G$, that is, $\{\tau_i, \sigma_t\} \notin E$. 

    \begin{adjustwidth}{2em}{0pt}
        \noindent \textbf{Claim:} For $i \leq t \leq k-1$, we have $\sigma_t < \tau_i$

        We prove this claim inductively. Suppose $\sigma_{k-1} > \tau_i$. Then, $\sigma_{k-1}\tau_i$ is a descent which is not a $G$-descent, contradicting the fact that $\sigma$ is a $G$-local minimum. 

         Now, suppose for $r < t \leq k-1$, we have $\sigma_t < \tau_i$. Then, if $\sigma_r > \tau_i$, we have $\sigma_r > \sigma_{r+1}$. As $\sigma$ is a $G$-local minimum, we have $\{\sigma_r, \sigma_{r+1}\} \in E$. Further, as $G$ is a natural unit interval graph, we get that $\{\tau_i, \sigma_r\} \in E$, which is a contradiction. Hence, $\sigma_r < \tau_i$. 

         Hence our claim is proved. 
    \end{adjustwidth}

     But then we have $\sigma_i < \tau_i$, which is a contradiction as $\tau$ is the lexicographic minimum, so $\tau_i < \sigma_i$. Hence, $\sigma = \tau$, that is, a $G$-local minimum in $S_R$ is the lexicographic minimum of $S_R$ and hence is unique. 
\end{proof}

Let $G$ be a natural unit interval graph. As the lexicographic minimum is unique for each region $R$ of $\mA_G$, we can choose the representative of the region $R$ to be the lexicographic minimum of $S_R$. By Lemma~\ref{locallex}, these lexicographic minima will be precisely the permutations where all descents are $G$-descents. 

\begin{lemma}\label{lexfaceproj}
    Let $v$ be as in Equation~\eqref{vec2}, and let $G = ([n],E)$ be a natural unit interval graph. Let $R$ be a region of $\mA_G$. Then $\proj_v(R) \in R_{\sigma}$ where $\sigma$ is the lexicographic minimum of $S_R$, and $R_{\sigma}$ is the region of the braid arrangement $\mA_n$ labeled by $\sigma$. 
\end{lemma}
\begin{proof}
    Let $\sigma' = \sigma_1'\ldots \sigma_n'\in S_R$ be such that $\proj_v(R) \in R_{\sigma'}$, where $R_{\sigma'}$ is the region of the braid arrangement $\mA_n$ labeled by $\sigma'$. Suppose $\sigma'$ is not the lexicographic minimum of $S_R$. Then, from Lemma~\ref{locallex}, $\sigma'$ is not a $G$-local minimum. Hence, for some $i \in [n]$, $\sigma_i' > \sigma_{i+1}'$ and $\{\sigma_i', \sigma_{i+1}'\} \notin E$. Then, $\sigma'' = \sigma_1'\ldots \sigma_{i+1}'\sigma_i'\ldots \sigma_n' \in S_R$. 

    Now, as a consequence of Theorem~\ref{graphSC} and Lemma~\ref{rlminsc}, the face $F'$ of $R_{\sigma'}$ that $v$ projects into is labeled by the ordered partition obtained by partitioning $\sigma'$ at the right-to-left minima. Let us denote this partition by $\Pi(F')$.  

    In $\Pi(F')$, $\sigma_i'$ and $\sigma_{i+1}'$ will be in the same block (as $\sigma_i'$ cannot be a right-to-left minimum). As swapping $\sigma_i'$ and $\sigma_{i+1}'$ does not affect the partition, $F'$ is a common face of $R_{\sigma'}$ and $R_{\sigma''}$. Hence $\proj_v(R) \in R_{\sigma''}$, the region of the braid arrangement $\mA_n$ labeled by $\sigma''$.  

    Continuing like this, we get that $\proj_v(R) \in R_{\sigma}$ where $\sigma$ is a $G$-local minimum of $S_R$ and hence the lexicographic minimum of $S_R$ by Lemma~\ref{locallex}.
\end{proof}

\begin{lemma}\label{faceexists}
    Let $G = ([n],E)$ be a natural unit interval graph, and let $R$ be a region of $\mA_G$. Let $\sigma\in \Sn_n$ be the lexicographic minimum of $S_R$. Let $\Pi$ be the ordered partition obtained by partitioning $\sigma$ at the right-to-left minima. Then, $F_{\Pi}$ is not contained in the relative interior of another face of $\mA_G$.
\end{lemma}
\begin{proof}
    To show that $F_{\Pi}$ is not contained in the relative interior of another face of $\mA_G$, it is enough to show that $\text{span($F_{\Pi}$)}$ is a flat of $\mA_G$. To show this, we show that for any block $B$ of $\Pi$, the induced subgraph $G[B]$ is connected. 
    
    Suppose the induced subgraph $G[B]$ is not connected. Let $i$ be the smallest vertex in $B$. As we get $\Pi$ by partitioning at the right-to-left minima, $i$ will appear after any other element of $B$ in the one line notation of $\sigma$. Let $j$ be the greatest vertex of the component with $i$, and $k$ be a vertex in a different component of the induced subgraph. 

    Now, as $i$ and $j$ are in the same component, we have a path from $i$ to $j$, say $ik_1k_2 \ldots k_t j$. Suppose we have $i < k < j$. Then, $k_s < k <k_{s+1}$ for some $s$ (take $k_0 = i$ and $k_{t+1} = j$). But, as $\{k_s, k_{s+1}\} \in E$, and $G$ is a natural unit interval graph, $\{k_s, k\} \in E$, that is, $k$ is in the same component as $k_s$ and hence the same component as $i$, a contradiction. 

    Hence $k > j$, that is, any other component of the induced subgraph has all vertices greater than $j$ and hence greater than any vertex in the component containing $i$. 
    We can order the components of $G[B]$ independently in the one line notation of $\sigma$ and still have a permutation in $S_R$. The permutation obtained by ordering these components such that the component with $i$ appears before any other component of $G[B]$ will be less than $\sigma$ in the lexicographic order. This contradicts the fact that $\sigma$ is the lexicographic minimum of $S_R$. Hence $G[B]$ is connected.      
    
\end{proof}
Hence, by Lemmas~\ref{lexfaceproj} and \ref{faceexists}, for a region $R \in R(\mA_G)$ and $v$ as in Equation~\eqref{vec2}, $\pd_v(R) = \text{dim($F_{\Pi}$)}$, where $\Pi$ is the ordered partition of $[n]$ obtained by partitioning the lexicographic minimum $\sigma$ of $S_R$ at its right-to-left minima. As $\dim(F_{\Pi})$ is clearly the number of right-to-left minima of $\sigma$, Theorem~\ref{NUIProjRL} follows.

We can now obtain the characteristic polynomial of the graphical arrangement of a natural unit interval graph. 
\begin{proof}[Proof of Corollary~\ref{NUIChar}]
    From Theorem~\ref{NUIProjRL}, we have for a region $R$ of the graphical arrangement $\mA_G$, $\pd_v(R) = \text{RLmin($\sigma$)}$ where $\sigma$ is the lexicographic minimum of $S_R$. We have also shown that these lexicographic minima are precisely the permutations where all descents are $G$-descents. 

    Hence, from Theorem~\ref{Kab} we have $$\chi_{\mA_G}(-q) = (-1)^n\sum\limits_{\substack{\sigma \in \Sn_n \\ \text{all descents of $\sigma$ are $G$-descents}} }q^{\text{RLmin($\sigma$)}}.$$

    So, it is enough to show that 
    
    $$\sum\limits_{\substack{\sigma \in \Sn_n \\ \text{all descents of $\sigma$ are $G$-descents}} }q^{\text{RLmin($\sigma$)}} = \prod\limits_{j = 1}^n (q + c_j).$$
    
    We prove this using induction on $n$. The base case is trivial.

    Suppose that the statement holds for every natural unit interval graph on $n-1$ vertices. 

    Let $G = ([n],E)$ be a natural unit interval graph and let $G'$ be the subgraph obtained by deleting the vertex $n$. Then $G'$ is also a natural unit interval graph. 

    We claim that if we have a permutation $\sigma$ of $[n]$ such that all descents are $G$-descents, removing $n$ from the $1$-line notation of $\sigma$ gives us a permutation of $[n-1]$ such that all descents are $G'$-descents. This is because for $\sigma = \sigma_1 \ldots \sigma_n$, with $\sigma_i = n$ ($i< n$), $n \sigma_{i+1}$ is a descent and hence a $G$-descent. If $\sigma_{i+1} < \sigma_{i-1}$, as $\{\sigma_{i+1}, n\} \in E$ and $G$ is a natural unit interval graph, we have $\{\sigma_{i+1},\sigma_{i-1}\} \in E$, and hence the only new descent formed on removing $n$ is a $G'$-descent. If $\sigma_{i+1} > \sigma_{i - 1}$, no new descent is formed on removing $n$ and hence all descents are $G'$-descents.  

    Hence, we can uniquely obtain all permutations of $[n]$  such that all descents are $G$-descents by inserting $n$ into a permutation of $[n-1]$ such that all descents are $G'$-descents. There are two ways to do this:

    We can insert $n$ before any element it is adjacent to so that the descent thus formed is a $G$-descent. We have $c_n$ ways of doing this. In this case, the number of right-to-left minima of the permutation will remain the same as $n$ does not become a right-to-left minima and the ordering of the other elements is unchanged. The contribution of this to the characteristic polynomial will be $c_n\prod\limits_{j = 1}^{n-1} (q + c_j)$.

    We can also insert $n$ at the end of the permutation. This will not form any new descents, hence all descents will be $G$-descents. However, the number of right-to-left minima will increase by one as $n$ is also a right-to-left minima. The contribution of this to the characteristic polynomial will be $q\prod\limits_{j = 1}^{n-1} (q + c_j)$.

    Hence, by induction, $\sum\limits_{\substack{\sigma \in \Sn_n \\ \text{all descents of $\sigma$ are $G$-descents}} }q^{\text{RLmin($\sigma$)}} = \prod\limits_{j = 1}^n (q + c_j)$.

    Hence, $\chi_{\mA_G}(-q) = (-1)^n\prod\limits_{j = 1}^n (q + c_j)$.
\end{proof}

\section{Type B Coxeter arrangement and signed graphs}\label{TypeBsec}
In this section, we show that for any subarrangement of the Type $B$ Coxeter arrangement, the projection statistic defined in Theorem~\ref{Kab} has a combinatorial interpretation that is a generalization of source components. We first establish some notation for this section.  

\begin{notation}
    Let $\npm := [n] \cup -[n]$. Then, for $i \in [n]$, and $x \in \R^n$, we denote $x_{-i} = -x_i$. 
    Further, for $m,n \in \N$, we denote $[-m\semicolon n] := \{-m, -(m-1), \ldots, n-1, n\}$.
\end{notation}

\begin{definition}
    The \emph{Type $B$ Coxeter arrangement} $\mB_n$ is defined as the following collection of hyperplanes: 
    $$\mB_n = \{H_{i,j} \mid -n \leq i < j \leq n,\,\, i,j \neq 0\},$$
    where $H_{i,j} = \{(x_1, \ldots, x_n) \in \R^n \mid x_i = x_j\}$, with $x_{-i} = -x_i$ for all $i \in [n]$.  
\end{definition}

\begin{figure}[ht]
    \centering
    \includegraphics[width=0.8\linewidth]{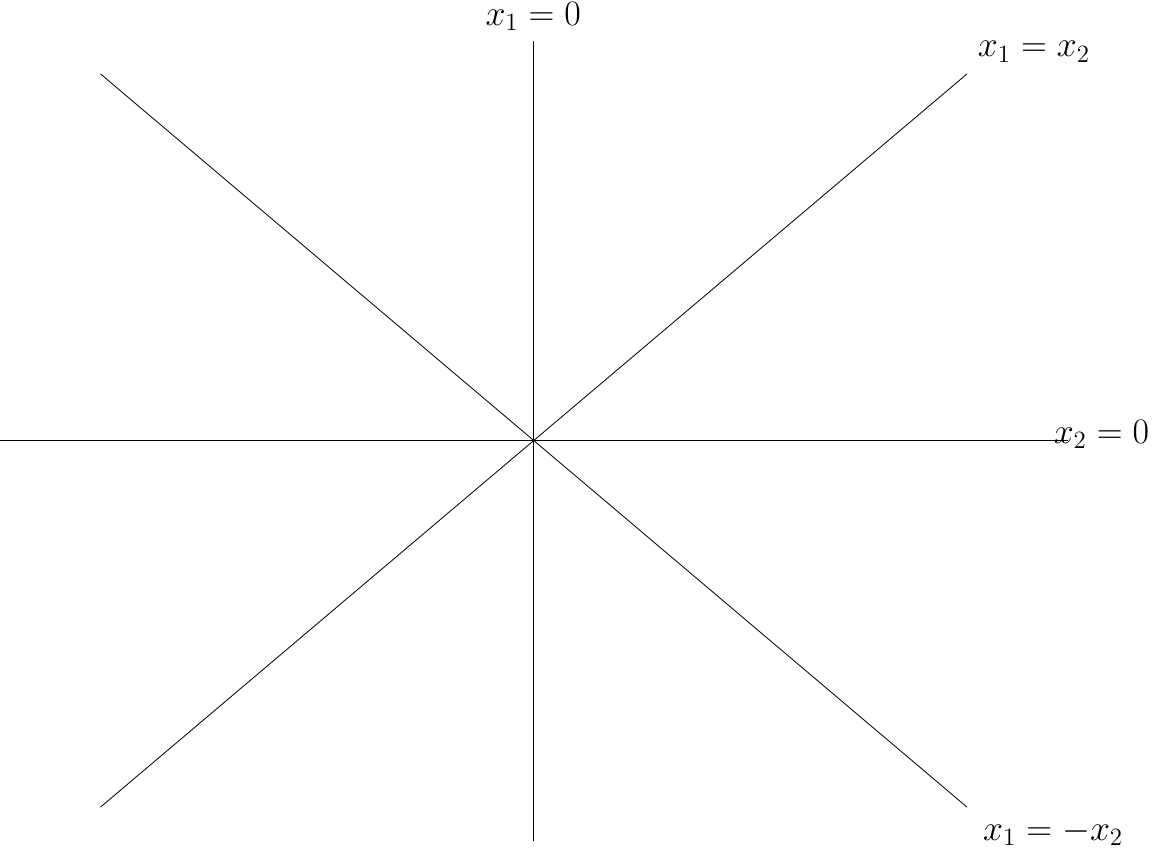}
    \caption{The two dimensional Type $B$ Coxeter arrangement $\mB_2$.}
    \label{fig:B2}
\end{figure}

\subsection{Symmetric graphs, orientations and signed source components}\hfill\\

In Section~\ref{gasec}, we defined subarrangements of the braid arrangement using simple graphs. We further related the chromatic polynomial of a simple graph to the characteristic polynomial of its corresponding arrangement. We can similarly define subarrangements of the Type $B$ Coxeter arrangement using \emph{symmetric graphs}. In this subsection, we will recall some definitions and results related to symmetric graphs, and state a generalization of Greene and Zaslavsky's statistic for coefficients of the chromatic polynomial of symmetric graphs. 

\begin{definition}
    Let $G = (\npm, E)$ be a graph. We say that $G$ is a \emph{symmetric graph} if for all $u, v \in \npm$, one has $\{u, v\} \in E$ if and only if $\{-u, -v\} \in E$. 
\end{definition}

\begin{definition}
    Let $G = (\npm, E)$ be a symmetric graph and let $k$ be a positive integer. A \emph{symmetric $k$-coloring} of $G$ is a function $c:\npm \ra [-k\semicolon k]$ such that $c(u) = -c(-u)$ for all $u \in [n]$. 
    We define the \emph{symmetric chromatic polynomial} as the polynomial $\chi_G^{\pm}(q)$ such that for all $k \in \N$, $$\chi_{G}^{\pm}(2k+1) = \text{ the number of proper symmetric $k$-colorings of $G$}. $$
\end{definition}

\begin{remark}\label{symmsigned}
     A \emph{signed graph} $\Sigma = (V, E', \sigma)$ is a graph $G = (V,E)$ together with a function $\sigma: E' \ra \{+, -\}$ assigning a sign to each edge. 
    To a symmetric graph $G = (\npm, E)$ one can associate the signed graph $\Sigma_G = ([n], E', \sigma)$ where 
    \begin{itemize}
        \item $e = \{u,v\} \in E'$ with $\sigma(e) = +$ if and only if $\{u,v\}$ and $\{-u,-v\} \in E$, and
        \item $e = \{u,v\} \in E'$ with $\sigma(e) = -$ if and only if $\{u,-v\}$ and $\{-u,v\} \in E$.
    \end{itemize}
    In fact, this defines a bijection between symmetric graphs on $\npm$ and signed graphs on $[n]$ where any loop is negative (as a positive loop would correspond to a loop in the symmetric graph). 

    Further, for a positive integer $k$, a \emph{proper $k$-coloring} of a signed graph $\Sigma = (V, E', \sigma)$ is a function $c: V \ra [-k\semicolon k]$ such that for an edge $e = \{u, v\} \in E'$, $c(u) \neq \sigma(e)c(v)$ and the \emph{signed chromatic polynomial} is the polynomial such that $$\chi_{\Sigma}(2k+1) = \text{ the number of proper $k$-colorings of $\Sigma$}.$$ 

    Given a proper symmetric $k$-coloring $c$ of a symmetric graph $G =(\npm, E)$, we can obtain a proper $k$-coloring $c'$ of the corresponding signed graph $\Sigma_G$ by defining $c'(u) = c(u)$ for all $u \in [n]$. In fact, this defines a bijection between proper symmetric $k$-colorings of a symmetric graph $G$ and proper $k$-colorings of the signed graph $\Sigma_G$. 

    Hence, $$\chi^{\pm}_G(q) = \chi_{\Sigma_G}(q).$$
\end{remark}

\begin{definition}
    Let $G = (\npm, E)$ be a symmetric graph. We say an orientation $\gamma$ of $G$ is \emph{symmetric} if for $u,v \in \npm$, we have $(u,v) \in \gamma$ if and only if $(-v, -u) \in \gamma$.  
\end{definition}

\begin{definition}
    Let $G = (\npm, E)$ be a symmetric graph and let $\gamma$ be an acyclic symmetric orientation of $G$. For $i \in [n]$, we define $\gamma_i$ to be the set of vertices reachable from $i$ by a directed path of $\gamma$ (with $i \in \gamma_i$), and $\gamma^{i}$ to be the set of vertices from which there is a directed path to $i$ (with $i \in \gamma^i$).  
\end{definition}

Note that for $i \in \npm$, $\gamma_i = -\gamma^{-i}$ as if $j$ is $i$-reachable via a directed path $iv_1v_2\ldots v_kj$ of $\gamma$, then $(-j)(-v_k)\ldots (-v_2)(-v_1)(-i)$ is also a directed path of $\gamma$ by symmetry of $\gamma$. So $-j \in -\gamma^{-i}$.

\begin{definition}\label{signedsourcecomp}
    Let $G = (\npm, E)$ be a symmetric graph and let $\gamma$ be an acyclic symmetric orientation of $G$. For $i \in [n]$, define $R_i = \begin{cases}
        \gamma_i & \text{ if } -i \notin \gamma_i \\
        \gamma_i \cap \gamma^{-i} & \text{ if } -i \in \gamma_i
    \end{cases}$, and set $B = \emp$. 
    We define $S_1, S_2\ldots$ recursively: for $k \geq 1$, if $\bigcup\limits_{i < k} (S_i \cup -S_i) \cup B = \npm$, then $S_k = \emp$. Otherwise, for $m = \min\left([n] \setminus \left(\bigcup\limits_{i < k} (S_i \cup -S_i) \cup B)\right)\right)$, 
    \begin{itemize}
        \item $B = B \cup R_m$ if $-m \in R_m$
        \item $S_k = R_m \setminus{\bigcup\limits_{i<k}S_i}$ if $-m \notin R_m$.
    \end{itemize}
     The non-empty subsets $S_k$ thus defined are the \emph{signed source components} of $\gamma$ and the final $B$ obtained is the \emph{zero block} of $\gamma$. 
\end{definition}

\begin{figure}[ht]
    \centering
    \includegraphics[width=0.85\linewidth]{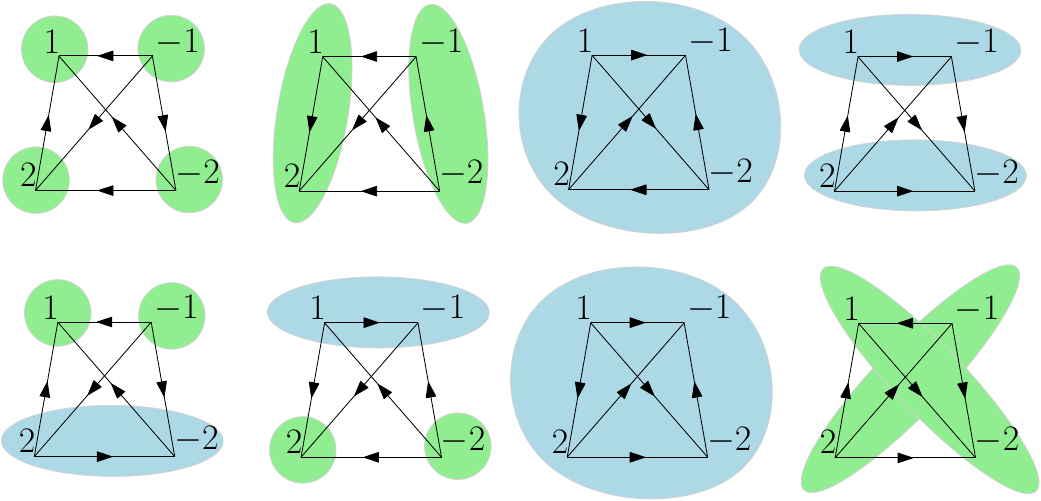}
    \caption{The signed source components of the $8$ acyclic orientations of the complete symmetric graph on $\pm[2]$. The signed source components $S_i$ and $-S_i$ are in green and the zero blocks are in blue.}
    \label{fig:SSC_Example}
\end{figure}

We will prove a generalization of Greene and Zaslavsky's Theorem~\ref{GZProj} for symmetric graphs by expressing the coefficients of $\chi_G^{\pm}(q)$ in terms of signed source components of symmetric acyclic orientations of the graph.
\begin{theorem}\label{SymmGZ}
    Let $G = (\npm, E)$ be a symmetric graph, and $k$ be a non-negative integer. Then, $$[q^k]\chi_G^{\pm}(q) = (-1)^{n-k}\beta_k$$ where $\beta_k$ is the number of symmetric acyclic orientations of $G$ with exactly $k$ signed source components. 
\end{theorem}

Theorem~\ref{SymmGZ} will be proved in Section~\ref{SecSymmProofs}. 

\begin{remark}
    Let $G = (\npm, E)$ be a symmetric graph and $\Sigma_G$ be the corresponding signed graph on $[n]$ as defined in Remark~\ref{symmsigned}. As $$\chi_{\Sigma_G}(q) = \chi_G^{\pm}(q), $$ Theorem~\ref{SymmGZ} gives a combinatorial interpretation of the coefficients of the chromatic polynomial of a signed graph with no positive loops. 
\end{remark}

\subsection{B-graphical arrangements}\hfill\\

In this section, we define subarrangements of the Type $B$ Coxeter arrangement and obtain a combinatorial interpretation for the coefficients of their characteristic polynomial via a projection statistic. 
\begin{definition}
    Let $G = (\npm, E)$ be a symmetric graph. The \emph{$B$-graphical arrangement} $\mB_G$ is defined as the collection of the following hyperplanes: 
    $$\mB_G := \{H_{i,j} \mid \{i,j\} \in E\},$$
    where $H_{i,j} = \{(x_1, \ldots, x_n) \mid x_i = x_j\}$, with $x_{-i} = -x_i$ for all $i \in [n]$.
\end{definition}

\begin{remark}
    The $B$-graphical arrangements are precisely the subarrangements of $\mB_n$. Hence, it is clear that subarrangements of $\mB_n$ are in bijection with symmetric graphs on $\npm$. 
\end{remark}

As with chromatic polynomials of graphs and characteristic polynomials of graphical arrangements, we have a similar relation between the symmetric chromatic polynomial of a symmetric graph and the characteristic polynomial of the corresponding $B$-graphical arrangement.
\begin{theorem}\cite{ZasSignedGraph}\label{ZasChromChar}
    Let $G = (\npm, E)$ be a symmetric graph and $\mB_G$ the corresponding $B$-graphical arrangement. Then, 
    $$\chi_{G}^{\pm} = \chi_{\mB_G}.$$  
\end{theorem}

Now, in order to obtain a combinatorial interpretation for the projection statistic, we need a labeling of the regions. As with graphical arrangements, we associate to each region an acyclic orientation of the underlying graph. 
\begin{figure}[ht]
    \centering
    \includegraphics[width=0.8\linewidth]{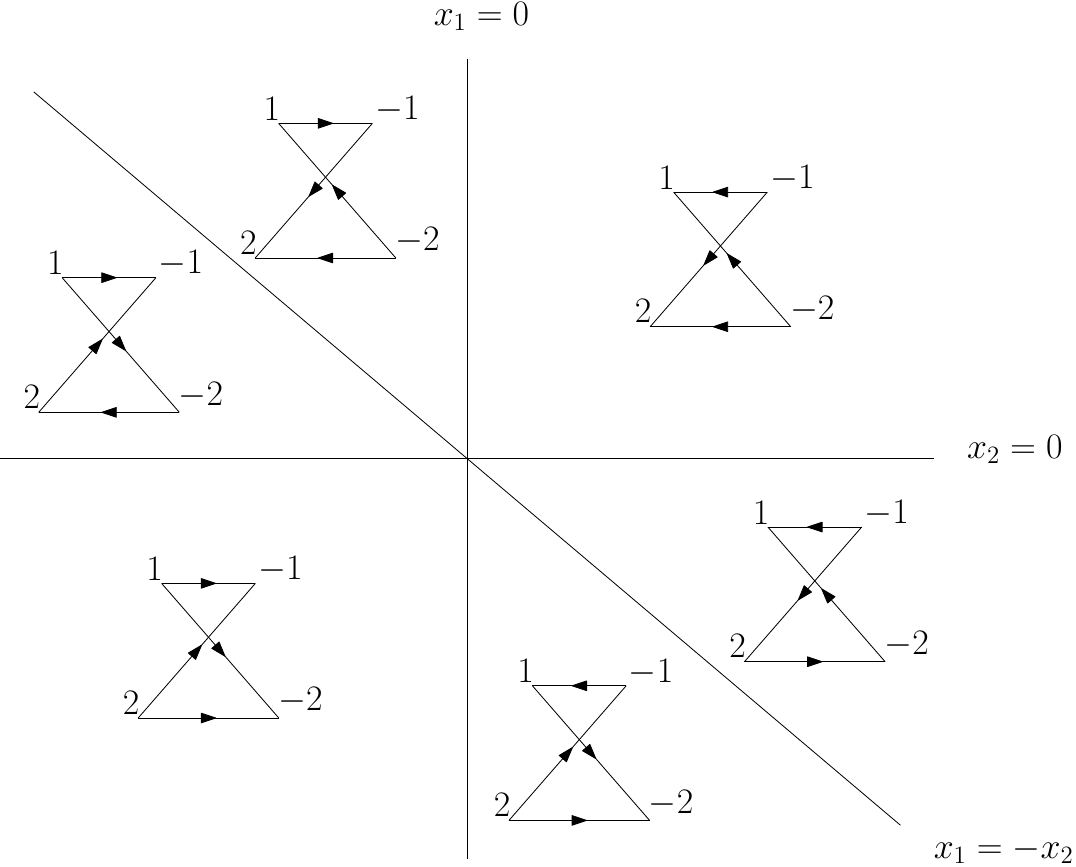}
    \caption{A $\mB$-graphical arrangement with the regions labeled by the symmetric acyclic orientations given by $\gamma$.}
    \label{fig:AcycSymmLab}
\end{figure}
\begin{definition}
    Let $R$ be a region of $\mB_G$. We associate to $R$ an acyclic orientation of $G$, denoted $\gamma_R$, by directing the edge $\{i, j\} \in E$ towards $i$ if and only if $x_i \geq x_j$ in $R$.  
\end{definition}

In fact, the acyclic orientations thus obtained are symmetric. See Fig.~\ref{fig:AcycSymmLab} for an example. 
The following result was given in~\cite{ZasSignedGraphOr}\footnote{The reference gives a bijection between regions of $\mB_G$ and acyclic orientations of the signed graph $\Sigma_G$. These acyclic orientations are in bijection with symmetric acyclic orientations of $G$. }. 
\begin{lemma}~\cite{ZasSignedGraphOr}
    The mapping $\gamma$ which associates to each region $R$ of $\mB_G$ the orientation $\gamma_R$ is a bijection between the set of regions of $\mB_G$ and the set of symmetric acyclic orientations of $G$. 
\end{lemma}

Next, we talk about the labels of the faces of the Type $B$ Coxeter arrangement. 
\begin{definition}
    Let $\Pi$ be a partition of $\npm$. We say $\Pi$ is a \emph{$\mB_n$-partition} if for any block $B$ of $\Pi$, its negative $-B$ is also a block of $\Pi$, and there is exactly one (possibly empty) block, called the \emph{zero block}, such that $B = -B$. 
    We say an ordered $\mB_n$-partition is \emph{antipalindromic} if it is of the form $(B_1,\ldots, B_k, B_0, B_{-k}, \ldots, B_{-1})$ where $B_{-i} = -B_i$ for all $i \in [k]$ and $B_0$ is the zero block. 
\end{definition}

Reiner showed in~\cite{Reiner} that the flats of the Type $B$ Coxeter arrangement are in bijection with the $\mB_n$-partitions of $\npm$, with the zero block indicating which coordinates are zero and the other blocks indicating which coordinates are equal (that is, if $i$ and $j$ are in the same block, then $x_i = x_j$ for any point $x$ of the flat with the convention that $x_{-i} = -x_i$ for all $i \in [n]$). 

Further, the faces of the Type $B$ Coxeter arrangement $\mB_n$ are in bijection with antipalindromic ordered $\mB_n$-partitions. Here, the relative order of the blocks indicates the relative order of the corresponding coordinates (that is, if $i$ is in a block before $j$, then $x_i > x_j$ for all points $x$ of the face). Note that the order of the blocks after $B_0$ is a mirror of the order before, so without loss of generality, we can consider $(B_1, \ldots, B_k, B_0)$ as the label of the face. Fig.~\ref{fig:TypeBFaces} demonstrates this labeling for the arrangement $\mB_2$.

\begin{figure}[ht]
    \centering
    \includegraphics[width=0.8\linewidth]{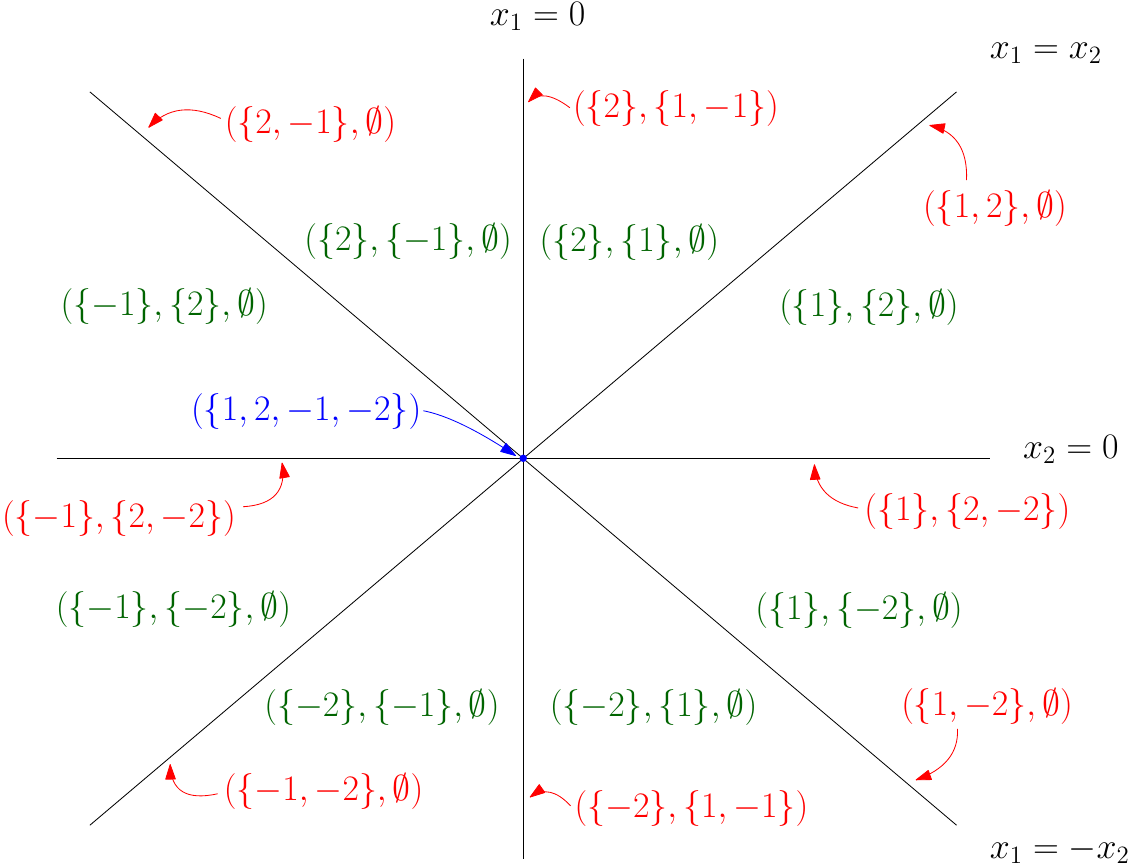}
    \caption{$\mB_2$ with the faces labeled with antipalindromic ordered $\mB_2$-paritions. The last set in the tuple is the (possibly empty) zero block. }
    \label{fig:TypeBFaces}
\end{figure}

\begin{definition}
    Let $G = (\npm, E)$ be a symmetric graph and $\gamma$ be a symmetric acyclic orientation of $G$. Let $S_1, \ldots, S_k$ be the signed source components of $\gamma$ and $B$ be the zero block of $\gamma$. Note that $B = -B$ and $(S_1, \ldots, S_k, B, -S_k, \ldots, -S_1)$ forms an antipalindromic $\mB_n$-partition of $\npm$. Let us denote this partition by $\Pi(\gamma)$ and the face labeled by it $F_{\Pi(\gamma)}$.  
\end{definition}

The following is the main result of this section. 
\begin{theorem}\label{typeBmain}
    Let $v = (v_1, \ldots, v_n) \in \R^n$ be such that
    \begin{equation}\label{vec_b}
        \forall i \in [n-1], \,\, v_i > (14n^2 + 1)v_{i+1} \text{ and } v_n > 0.
    \end{equation}
    Let $G = (\npm, E)$ be a symmetric graph, let $R$ be a region of the $B$-graphical arrangement $\mB_G$, and let $\gamma_R$ be the symmetric acyclic orientation labeling $R$. Suppose $\gamma_R$ has $k$ signed source components. 

    Then, $\pd_v(R) = k$, that is, the projection dimension of $v$ on $R$ equals the number of signed source components of $\gamma_R$. In fact, $\proj_v(R)$ lies in the relative interior of the face $F_{\Pi(\gamma_R)}$. 
\end{theorem}
\begin{figure}[ht]
    \centering
    \includegraphics[width=0.8\linewidth]{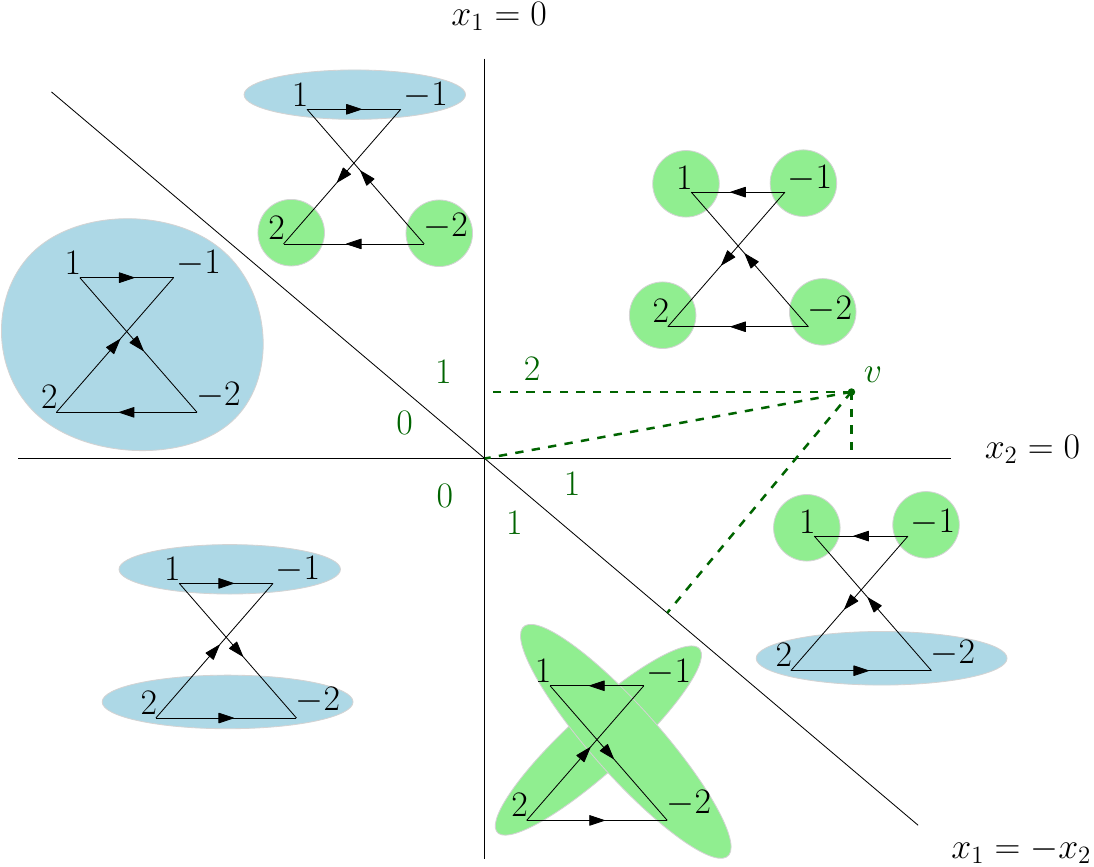}
    \caption{A $\mB$-graphical arrangement with projection from a point $v$ as in Equation~\eqref{vec_b} onto the regions, with the regions labeled by symmetric acyclic orientations with signed source components (signed source components both $S_i$ and $-S_i$ in green, zero block $B$ in blue) and projection dimension in green. }
    \label{fig:BGraphProj}
\end{figure}

The following corollary is a simple consequence of Theorems~\ref{Kab} and~\ref{typeBmain}.
\begin{corollary}\label{BCharCoeffs}
    Let $G = (\npm, E)$ be a symmetric graph and let $\mB_G$ be the corresponding $B$-graphical arrangement. Then, $$[q^k]\chi_{\mB_G}(q) = (-1)^{n-k}\beta_k$$ where $\beta_k$ is the number of symmetric acyclic orientations of $G$ with $k$ signed source components.
\end{corollary}

\begin{remark}
    Note that the Type $A$, Type $B$ and Type $D$ Coxeter arrangements are all subarrangements of the Type $B$ Coxeter arrangement and hence they are $B$-graphical arrangements. Hence Corollary~\ref{BCharCoeffs} gives us a combinatorial interpretation for the coefficients of the characteristic polynomials of these arrangements. 

    There are other combinatorial interpretations of these coefficients. For example, in~\cite{DM}, Deshpande and Menon give an interpretation that is a generalization of the right-to-left minima statistic on signed permutations. However, unlike for the braid arrangement, where the RLmin statistic coincides with the number of source components, Deshpande and Menon's statistic does not coincide with our generalization to signed source components.  
\end{remark}

\subsection{Proofs of Theorem~\ref{SymmGZ} and Theorem~\ref{typeBmain}}\label{SecSymmProofs}\hfill
\begin{remark}\label{GZGenRemProof}
    Note that Theorem~\ref{SymmGZ} is a direct consequence of Theorem~\ref{ZasChromChar} and Corollary~\ref{BCharCoeffs}.
\end{remark}

Now, before we proceed to the proof of Theorem~\ref{typeBmain}, we need to prove some lemmas that allow us to characterize the faces that the projection from a point $v$ lies in the relative interior of. 
\begin{lemma}\label{typeBproj}
    Let $v = (v_1, \ldots, v_n) \in \R^n$ and $\Pi$ be a $\mB_n$-partition of $\npm$ labeling a flat $X$ of $\mB_n$. For $i \in [n]$, let $B$ be the block of $\Pi$ containing $i$. Then the $i^{th}$ coordinate of the projection of $v$ on $X$ is given by $\dfrac{\sli_{j \in B}v_j}{|B|}$, where $\forall j \in [n], \, v_{-j} = -v_j$. 
\end{lemma}
\begin{proof}
    Recall that $\proj_v(X)$ is the orthogonal projection of $v$ on $X$. 

    Let $\mathbb{B} = \{B_1, \ldots, B_k\}$ be such that $B \in \mathbb{B}$, for all non-zero blocks $B'$ of $\Pi$, exactly one of $B'$ and $-B'$ is in $\mathbb{B}$ and $\mathbb{B}$ contains the zero block if and only if it is non-empty. Without loss of generality, let $B_1 = B$. 

    Now, for $j \in [k]$, define $f_j = \sli_{i \in B_j} e_i$ where $\forall i \in [n]$, $e_i$ is the $i^{th}$ standard basis vector of $\R^n$ and $e_{-i} = -e_i$. Note that the $i^{th}$ coordinate of $f_j$ for $j \neq 1$ is $0$ and the $i^{th}$ coordinate of $f_1$ is $1$ (as $i \in B_1$). 

    Then two cases arise. 

    \begin{adjustwidth}{2em}{0pt}
        \textbf{Case 1:} $B$ is not the zero block. 

        Then, $O = \{f_j\}_{j \in [k]}$ forms an orthogonal basis of $X$. 
        Hence, $\proj_v(X) = \sli_{j \in [k]} \dfrac{\langle v, f_j\rangle}{||f_j||^2}f_j =  \dfrac{\sli_{j \in B}v_j}{|B|}.$

        \noindent \textbf{Case 2:} $B$ is the zero block. 

        Then, $f_1 = 0$, and $O = \{f_j\}_{1 < j \leq k}$ forms an orthogonal basis of $X$. 

        Hence, $\proj_v(X) = \sli_{j \in [k]} \dfrac{\langle v, f_j\rangle}{||f_j||^2}f_j = 0 =   \dfrac{\sli_{j \in B}v_j}{|B|}.$
    \end{adjustwidth}

    Hence the result holds. 
\end{proof}

\begin{definition}
    Let $\Pi$ be a $\mB_n$-partition of $\npm$ and $B$ be a block that is not the zero block. We define 
    $$\absmin(B) := i \in B \text{ such that } |i| \leq |j| \text{ for all } j \in B.$$
\end{definition}
Note that as $B$ is not the zero block, it contains at most one of $i$ or $-i$ for all $i \in [n]$, and hence $\absmin$ is well defined. 

The following lemma characterizes the $v$-faces (see Definition~\ref{goodface}) of $\mB_n$.

\begin{lemma}\label{typeBgoodface}
    Let $v \in \R^n$ be such that $v_i > 2nv_{i+1}$ for all $i \in [n-1]$ and $v_n >0$. Let $F$ be a face of $\mB_n$ labeled by $(B_1, \ldots, B_k, B_0)$. Then $F$ is a $v$-face of $\mB_n$ if and only if  $0< \absmin(B_i) < \absmin(B_{i+1})$ for all $i \in [k-1]$. 
\end{lemma}
\begin{proof}
    Let $X = \s(F)$, that is, $X$ is the flat labeled by the unordered partition $\Pi = \{B_1, \ldots, B_k, B_0, B_{-1}, \ldots, B_{-k}\}$. Let $\proj_v(X) = (p_1, \ldots, p_n)$. 

    Let $F$ be a $v$-face of $\mB_n$. We first show that any non-zero block $B$ of $\Pi$ with $\absmin(B)>0$ appears before $B_0$ in the ordered partition labeling $F$. Next, we show that given two non-zero blocks $B$ and $B'$ of $\Pi$ with $0 < \absmin(B) < \absmin(B')$, $B$ must appear before $B'$ in the ordered partition of $F$.

    Let $B =  \{i_1, \ldots i_s, -j_1, \ldots, -j_t\}$ be a non-zero block of $\Pi$ such that $\absmin(B) = i_1 >0$. Then, from Lemma~\ref{typeBproj}, as $v_{k} \geq v_n$ for all $k \in [n]$, and as $j_k \geq i_1 + 1$, $v_{j_k} \leq v_{i_1 + 1}$ for all $k \in [t]$, we have $$p_{i_1} = \dfrac{v_{i_1} + \ldots v_{i_s} - v_{j_1} - \ldots - v_{j_t}}{s+t} \geq \dfrac{v_{i_1} + (s-1)v_n - tv_{i_1 + 1}}{s+t}.$$
    Now, as $v_{i_1} > 2nv_{i_1 + 1}$, we have $p_{i_1} > 0$. Hence, any block $B$ of $\Pi$ such that $\absmin(B) > 0$ appears before $B_0$ in the ordered partition labeling $F$. 

    Let $B' = \{i_1', \ldots i_{s'}', -j_1', \ldots, -j_{t'}'\}$ be another non-zero block of $\Pi$ such that $\absmin(B') = i_1'$. Further, we assume that $0 < i_1 < i_1'$.
    
    Then, from Lemma~\ref{typeBproj}, as $v_{i_k'} \leq v_{i_1'}$ for all $k \in [s']$, and as $i_1 < i_1'$, 
    \begin{align*}
        p_{i_1'} = \dfrac{v_{i_1'}+ \ldots + v_{i_{s'}'} - v_{j_1'} - \ldots - v_{j_{t'}'}}{s'+t'} \leq v_{i_1'} \leq v_{i_1 + 1}.
    \end{align*}
    Hence, 
    \begin{align*}
        p_{i_1} - p_{i_1'} \geq \dfrac{v_{i_1} + (s-1)v_n - (s+2t) v_{i_1 + 1}}{s+t} \geq \frac{(2n -s -2t) v_{i_1 + 1} + (s-1) v_n}{s+t} > 0.   
    \end{align*}
    Hence $B$ must appear before $B'$ in the ordered partition.

    Hence we have shown that if $F$ is a $v$-face of $\mB_n$, then $0 < \absmin(B_i) < \absmin(B_{i+1})$ for all $i \in [k-1]$. Conversely, let $F$ be a face of $\mB_n$ labeled by $(B_1, \ldots, B_k, B_0)$ such that $0 < \absmin(B_i) < \absmin(B_{i+1})$ for all $i \in [k-1]$. Note that for $X = \s(F)$, exactly one face of $X$ satisfies this condition. Further, the projection of $v$ on $X$ lies in a unique face of $X$ which satisfies this condition. Hence, the projection of $v$ on $X$ must lie in $F$, that is, $F$ is a $v$-face of $\mB_n$. 
\end{proof}

\begin{lemma}\label{typeBcontainmentlemma}
    Let $G = (\npm, E)$ be a symmetric graph, let $R$ be a region of $\mB_G$ and let $\gamma_R$ be the symmetric acyclic orientation of $G$ labeling $R$. Let $\Pi(\gamma_R) = (B_1, \ldots, B_k, B_0)$. Then, for $v = (v_1, \ldots, v_n)\in \R^n$ such that $v_i > 2nv_{i+1}$ for all $i \in [n-1]$, $F_{\Pi(\gamma_R)}$ is a $v$-face of $R$. Moreover, if $\Pi' = (D_1, \ldots, D_{\ell}, D_0)$ is such that $F_{\Pi'}$ is a $v$-face of $R$, we have for all $j \in [k]$,  
    \begin{equation}\label{typeBcontainmenteq}
        \bigcup\limits_{i = 1}^j B_i \sse \bigcup\limits_{i = 1}^j D_i.
    \end{equation}
\end{lemma}
\begin{proof}
    Let $\Pi$ be an antipalindromic ordered $\mB_n$-partition of $\npm$. Then the flat $\s(F_{\Pi})$ of $\mB_n$ is a flat of $\mB_G$ if and only if
    \begin{itemize}
        \item for every non-zero block $B$ of $\Pi$, the induced subgraph $G[B]$ is connected
        \item there is a path between $i$ and $-i$ for any $i$ in the zero block $B_0$.
    \end{itemize}
    Clearly, by the definition of signed source components, this holds for $\Pi(\gamma_R)$. 
    
    Further, $F_{\Pi}$ is a face of a region $R$ if and only if any weak inequality that holds in $F_{\Pi}$ holds in $R$. Hence it is easy to see that $F_{\Pi(\gamma_R)}$ is a face of $R$. 
    Further, by definition of signed source components, $0< \absmin(B_i) < \absmin(B_{i+1})$ for all $i < k$. Hence $F_{\Pi(\gamma_R)}$ is a $v$-face of $R$. 
    
    The proof of Equation~\eqref{typeBcontainmenteq} follows the same lines as that of Equation~\eqref{containment}.
\end{proof}

We can now complete the proof of Theorem~\ref{typeBmain}. 
\begin{proof}[Proof of Theorem~\ref{typeBmain}]
    By Lemma~\ref{typeBcontainmentlemma} we know that $F:= F_{\Pi(\gamma_R)}$ is a $v$-face of $R$. Hence it suffices to show that for any $v$-face $F' \neq F$ of $R$, one has $||v - \proj_v(F)|| < ||v - \proj_v(F')||$.

    Let $\Pi(\gamma_R) = (B_1, \ldots, B_k, B_0)$, and let $\absmin(B_i) = b_i$. By the definition of source components, we have $0 < b_1 < \ldots < b_k$.
    
    Let $F'$ be another $v$-face of $R$, and let $(D_1, \ldots, D_{\ell}, D_0)$ be its label satisfying $\absmin(D_i) < \absmin(D_{i+1})$ for all $i < \ell$. Now, suppose that $F'$ is the face of $R$ that $v$ projects into.

    \begin{adjustwidth}{2em}{0pt}
        \textbf{Claim:} Each of $D_1, \ldots, D_{\ell}$ contains at most one of the $b_i$'s.

        Suppose not. Then, for some $t \in [\ell]$, some $D_t$ contains $b_i$ and $b_{i'}$ with $i'> i$.  Let $D_t'$ be the set of all $b_{i'}$-reachable elements of $D_t$ (including $b_{i'}$). Let $\Pi_1 = (D_1, \ldots, D_{t}\setminus D_{t}', D_t', \ldots, D_{\ell}, D_0)$ correspond to the face $F_1$ of $R$. Clearly, $F_1$ is another face of $R$ of higher dimension containing $F$. Now, suppose we have $b_i < \absmin(D_j) < b_{i'}$ for $j \in [t+1; r]$ and $\absmin(D_{r+1}) > b_{i'}$. Let $\Pi_2 = (D_1, \ldots, D_t\setminus D_t', D_{t+1}, \ldots, D_r, D_t', \ldots, D_{\ell}, D_0)$ correspond to the face $F_2$ of $\mB_G$. Then $F_2$ is a $v$-face of $\mB_G$. 

        \begin{adjustwidth}{2em}{0pt}
            \textbf{Subclaim:} $\Pi_1$ and $\Pi_2$ correspond to the same face of $\mB_G$. 

            Let $j \in [t+1; r]$ and let $\al \in D_j$. Then, for any $\beta \in D_t'$, we must have that there is no directed path from $\beta$ to $\al$ in $\gamma_R$. If there were such a path, we would have $x_{\al} \geq x_{\beta}$ for all $x \in R$, and hence $D_j \sse D_t$ or $D_j$ must appear before $D_t$ in $\Pi_2$, which is not true. Hence, $D_t'$ and $D_j$ are independent for all $j \in [t+1;r]$ and any ordering of these blocks corresponds to the same face of $\mB_G$. 
        \end{adjustwidth}
    \end{adjustwidth}

    Hence, $F_2$ is a $v$-face of $R$ of higher dimension containing $F$. The projection of $v$ onto this face must be of shorter length than the projection onto $F'$, which is a contradiction. Hence $\ell \geq k$. Further, from Lemma~\ref{typeBcontainmentlemma}, we must have that for all $i \in [k]$, $b_i \in D_i$.

    Now, from Equation~\eqref{typeBcontainmenteq}, $\exists j \in [k]$ such that $D_j \supsetneq B_j$ and $D_i = B_i$ for all $i < j$. Let $j$ be the first index where $D_j \supsetneq B_j$, and let $|D_j| = d$, $|B_j| = b$. 

    From Lemma~\ref{typeBproj}, we have $p_B = \proj_v(F_{\Pi(\gamma_R)}) = (p_1, \ldots, p_n)$, where $\forall i \in B_j$, $p_i = \dfrac{\sli_{m \in B_j} v_m}{|B_j|}$, and $p_D = \proj_v(F') =  (q_1, \ldots, q_n)$ where $\forall i \in D_j$, $q_i = \dfrac{\sli_{m \in D_j} v_m}{|D_j|}$, with the usual convention $v_{-i} = -v_i$ for all $i \in [n]$. 

    It now suffices to show $||v - p_B|| < ||v - p_D||$.

    Then, 
    \begin{align*}
        ||v - p_B||^2 &=  \sli_{i = 1}^n (v_i - p_i)^2 \\
        &=  \sli_{\substack{i \in B_m \\0<m < j}} (v_i - p_i)^2 + \sli_{i \in B_j}(v_i - p_i)^2 + \sli_{\substack{i \in B_m\\ m>j}} (v_i - p_i)^2 + \sli_{i \in B_0} (v_i - p_i)^2.
    \end{align*}
    Now,
    \begin{align*}
        (v_{b_j} - p_{b_j})^2 = \left(\dfrac{(b-1)}{b}v_{b_j} - \dfrac{1}{b}\sli_{\substack{i \in B_j \\ i \neq b_j}} v_i\right)^2 = \left(\dfrac{b-1}{b}\right)^2v_{b_j}^2 + \eps_1,
    \end{align*}
    where
    \begin{align*}
        \eps_1 = \left(\dfrac{1}{b}\sli_{\substack{i \in B_j \\ i \neq b_j}}v_i\right)^2 - \dfrac{2(b-1)}{b}v_{b_j}\left(\dfrac{1}{b}\sli_{\substack{i \in B_j \\ i \neq b_j}}v_i\right). 
    \end{align*}
    Now, as $v_i \geq (14n^2 + 1)v_{i+1}$ for all $i < n$, 
    \begin{align*}
        \left|\frac{1}{b}\sli_{\substack{i \in B_j \\ i \neq b_j}} v_i\right| &\leq \frac{1}{b}\left(\frac{1}{14n^2 + 1}v_{b_j} + \frac{1}{(14n^2 + 1)^2}v_{b_j} + \ldots + \frac{1}{(14n^2 + 1)^{b-1}}v_{b_j} \right) \\
        &\leq \frac{1}{14n^2b}v_{b_j}.
    \end{align*}
    Hence,
    \begin{align*}
        |\eps_1| \leq \frac{2(b-1)}{14n^2b^2}v_{b_j}^2 + \frac{1}{196n^4b^2}v_{b_j}^2 \leq \frac{2b - 1}{14n^2b^2}v_{b_j}^2.
    \end{align*}
    Further, 
    \begin{align*}
        \sli_{\substack{i \in B_j \\ i \neq b_j } }(v_i - p_i)^2 = \sli_{\substack{i \in B_j \\ i \neq b_j}}\left(\dfrac{v_{b_j}}{b} + \left(\sli_{\substack{\ell \in B_j \\ \ell \neq b_j}}\dfrac{v_{\ell}}{b} - v_i\right)\right)^2 = \dfrac{(b-1)}{b^2}v_{b_j}^2 + \eps_2,
    \end{align*}
    where
    \begin{align*}
        \eps_2 = \sli_{\substack{i \in B_j \\ i \neq b_j}}\dfrac{2v_{b_j}}{b}\left(\sli_{\substack{\ell \in B_j \\ \ell \neq b_j}}\dfrac{v_{\ell}}{b} - v_{i}\right)
        + \sli_{\substack{i \in B_j \\ i \neq b_j}}\left(\sli_{\substack{\ell \in B_j \\ \ell \neq b_j}}\dfrac{v_{\ell}}{b} - v_{i}\right)^2. 
    \end{align*}
    Now, as $v_i \geq (14n^2 + 1)v_{i+1}$ for all $i < n$,
    \begin{align*}
        \left|\sli_{\substack{\ell \in B_j \\ \ell \neq b_j}}\frac{v_{\ell}}{b} - v_i\right| \leq \frac{2}{14n^2 + 1}v_{b_j},
    \end{align*}
     and hence,
    \begin{align*}
        |\eps_2| \leq \frac{4(b-1)}{b(14n^2 + 1)}v_{b_j}^2 + \frac{2(b-1)}{(14n^2 + 1)^2}v_{b_j}^2.
    \end{align*}
    Finally, let \begin{align*}
        \eps_3 = \sli_{i \in B_m, m > j} (v_i - p_i)^2.
    \end{align*}
    Now, for $i \in B_m$, $m > j$,
    \begin{align*}
        |v_i - p_i| \leq \frac{2}{14n^2 + 1}v_{b_j},
    \end{align*}
    Hence,   
    \begin{align*}
        |\eps_3| = \left|\sli_{i \in B_m, m > j} (v_i - p_i)^2\right| \leq \frac{4(n-b)}{(6n^2 + 1)^2}v_{b_j}^2.
    \end{align*}
    Now, for $\eps = \eps_1 + \eps_2 + \eps_3$,
    \begin{align*}
        |\eps| &\leq |\eps_1| + |\eps_2| + |\eps_3| \\
        &\leq \left(\frac{2b-1}{14n^2b^2} + \frac{4(b-1)}{b(14n^2 + 1)} + \frac{2(b-1)}{(14n^2 + 1)^2} + \frac{4(n-b)}{(14n^2 + 1)^2}\right)v_{b_j}^2 \\
        &\leq \frac{1}{14n^2}\left(\frac{2b-1}{b^2} + \frac{4(b-1)}{b} + \frac{4n -2b -2}{14n^2 + 1}\right)v_{b_j}^2 \\
        &\leq \frac{1}{2n^2}v_{b_j}^2.
    \end{align*}  
    Then, two cases arise: 
    \begin{adjustwidth}{2em}{0pt}
    \textbf{Case 1}: $B_0 = D_0$.
    
    Then, 
    \begin{align*}
        ||v - p_B||^2 
        \leq \sli_{i \in B_m, m < j} (v_i - p_i)^2 + \left(1 - \dfrac{1}{b}\right)v_{b_j}^2 + \dfrac{1}{2n^2}v_{b_j}^2.
    \end{align*}
    
    Similarly,  
    \begin{align*}
        ||v - p_D||^2 
        \geq \sli_{i \in D_m, m < j} (v_i - q_i)^2 + \left( 1- \dfrac{1}{d}\right)v_{b_j}^2 - \dfrac{1}{2n^2}v_{b_j}^2.
    \end{align*}
    
    Now, as $B_i = D_i$ for all $i < j$, and $d > b$, 
    \begin{align*}
        ||v - p_D||^2 - || v - p_B||^2 \geq \left(\dfrac{1}{b} - \dfrac{1}{d} - \dfrac{1}{n^2}\right)v_{b_j}^2 > 0. 
    \end{align*}
    \end{adjustwidth}

    \begin{adjustwidth}{2em}{0pt}
        \textbf{Case 2}: $B_0 \neq D_0$.
        \begin{adjustwidth}{2em}{0pt}
        \noindent Then, as a consequence of Equation~\eqref{typeBcontainmenteq}, we have $D_0 \subsetneq B_0$. 

        \noindent Let $t$ be the least index such that $D_t$ contains an element of $B_0$, and let the least such element in $D_t$ be $\beta$. Further, let $\absmin(D_t) = d_t$.

        \noindent Now, as $\beta$ is an element of $B_0$, there exists an index $i \in B_0$, $i > 0$ such that $\beta$ lies on a directed path from $i$ to $-i$. Let $\alpha$ be the least such index. Then, as $-\alpha$ is $\beta$-reachable, we have $-\alpha \in D_t$. In fact, as $-\alpha \in B_0$, and $\alpha \leq |\beta|$, we have $\beta = -\alpha$. Note that we cannot have $|\beta| < d_t$ as then $\absmin(D_t) = \beta < 0$, contradicting the fact that $F'$ is a $v$-face.  

        \noindent Further, any element of $B_0 \setminus D_0$ must be greater than $d_t$ in absolute value as if not, we will have $\al' \in D_s$ for some $s>t$, with $|\al'| < d_t$. Then $\absmin(D_s) < \absmin(D_t)$, contradicting that $F'$ is a $v$-face. 

        \noindent Now, 
        \begin{align*}
            (v_{\beta} - q_{\beta})^2 = \left(\frac{v_{d_t}}{|D_t|} + \frac{\sli_{\substack{i \in D_t\\ i \neq d_t}}v_i}{|D_t|} - v_{|\beta|}\right)^2 \geq \left(\frac{v_{d_t}}{2|D_t|}\right)^2, 
        \end{align*}
        and, as $p_i = 0$ for $i \in B_0$,
        \begin{align*}
            \sli_{i \in B_0 \setminus D_0} (v_i - p_i)^2 = \sli_{i \in B_0 \setminus D_0} v_i^2 \leq \frac{v_{d_t}^2}{14n^2} \leq \left(\frac{v_{d_t}}{2|D_t|}\right)^2.
        \end{align*}
        
        \noindent Hence, 
        \begin{align*}
            ||v - p_{B}||^2 \leq \sli_{\substack{i \in B_m\\ 0<m < j}} (v_i - p_i)^2 + &\left(1 - \dfrac{1}{b}\right)v_{b_j}^2 + \dfrac{1}{2n^2}v_{b_j}^2 \\&+ \left(\frac{v_{d_t}}{2|D_t|}\right)^2 + \sli_{i \in D_0} (v_i - p_i)^2.
        \end{align*}
        and, 
        \begin{align*}
            ||v - p_D||^2 
        \geq \sli_{\substack{i \in B_m\\ 0<m < j}} (v_i - q_i)^2 + &\left( 1- \dfrac{1}{d}\right)v_{b_j}^2 - \dfrac{1}{2n^2}v_{b_j}^2 \\&+ \left(\frac{v_{d_t}}{2|D_t|}\right)^2 + \sli_{i \in D_0} (v_i - q_i)^2.
        \end{align*}
        Now, as $B_i = D_i$ for all $i \in [j-1]$, $p_i = q_i = 0$ for $i \in D_0$, and $d>b$, 
        \begin{align*}
            ||v - p_D||^2 - || v - p_B||^2 \geq \left(\dfrac{1}{b} - \dfrac{1}{d} - \dfrac{1}{n^2}\right)v_{b_j}^2 > 0. 
        \end{align*}
        \end{adjustwidth}
    \end{adjustwidth}
\end{proof}

\begin{remark}
    Let $G = ([n], E)$ be a graph. Note that the graphical arrangement $\mA_G$ is a subarrangement of the Type $B$ Coxeter arrangement. In fact, for the graph $G^{\pm}$ on $\npm$ defined as $G \cup -G$, where $-G = (-[n], -E)$, $\mA_G$ is precisely the $B$-graphical arrangement $\mB_{G^{\pm}}$. Further, the signed source components of any symmetric acyclic orientation of $G^{\pm}$ are precisely the source components of the corresponding acyclic orientation on $G$. Hence, if we restrict Theorems~\ref{SymmGZ} and ~\ref{typeBmain} to symmetric graphs corresponding to graphical arrangements, we recover Greene and Zaslavsky's result (Theorem~\ref{GZProj}) and Theorem~\ref{graphSC} respectively. 
\end{remark}

\section*{Acknowledgements}
I owe many thanks to Theo Douvropoulos and Olivier Bernardi for supplying the initial problem. I thank Theo Douvropoulos for supplying some initial results (Theorem \ref{braidRLmin} and Lemmas \ref{projchar} and \ref{GFCond}), and his guidance and suggestions. I also thank Olivier Bernardi for suggesting the extension to natural unit interval graphs and being a constant source of guidance throughout. In addition, I thank Vasiliy Neckrasov for his help with Lemma \ref{lbdim}.

\printbibliography

\end{document}